\documentclass[a4paper,12pt]{amsart}
\usepackage{amsmath,amsfonts,amssymb}
\usepackage{amsthm}
\usepackage{amstext,amscd,mathabx,mathdots}
\usepackage{graphics,graphicx,epsf}
\usepackage[usenames]{color}
\usepackage{hyperref}
\usepackage{tikz}
\usepackage{xcolor}
\usepackage{color}
\usepackage{float}
\usetikzlibrary{arrows}

\newtheorem{theorem}{Theorem}[section]
\newtheorem{definition}[theorem]{Definition}

\newtheorem{lemma}[theorem]{Lemma}
\newtheorem{corollary}[theorem]{Corollary}
\newtheorem{remark}[theorem]{Remark}

\renewcommand{\proof}{{\noindent \bf Proof:\ }}

\numberwithin{equation}{section}

\title[Fractional powers approach of operators]{Fractional powers approach of operators for higher order abstract Cauchy problems} 

\author[F. D. M. Bezerra]{Flank D. M. Bezerra}
\address[F. D. M. Bezerra]{Departamento de Matem\'atica, Universidade Federal da Para\'iba, 58051-900 Jo\~ao Pessoa PB, Brazil.}
\email{flank@mat.ufpb.br}

\author[L. A. Santos]{Lucas A. Santos}
\address[L. A. Santos]{Departamento de Matem\'atica, Universidade Federal da Para\'iba, 58051-900 Jo\~ao Pessoa PB, Brazil.}
\email{lucas92mat@gmail.com}

\date{\today}

\begin{document}

\maketitle

\begin{abstract}
In this paper we explore the theory of fractional powers of non-negative (and not necessarily self-adjoint) operators and its amazing relationship with the Chebyshev polynomials of the second kind to obtain results of existence, regularity and behavior asymptotic of solutions for linear abstract evolution equations of $n$-th order in time, where $n\geqslant3$. We also prove generalizations of classical results on structural damping for linear systems of differential equations. 

\vskip .1 in \noindent {\it Mathematics Subject Classification 2010}: 
\newline {\it Key words and phrases:} sectorial operator; fractional powers; analytic semigroup; Chebyshev polynomials.

\end{abstract}

\tableofcontents

\section{Introduction} 

In this paper we consider the  following abstract linear evolution equation of $n$-th order in time 
\begin{equation}\label{Eq1}
\dfrac{d^nu}{dt^n} +Au = 0,\quad t>0,
\end{equation}
with initial conditions given by
\begin{equation}\label{Eq1ma}
u(0)=u_0\in X^{\frac{n-1}{n}},\ \dfrac{du}{dt}(0)=u_1\in X^{\frac{n-2}{n}},\ \dfrac{d^2u}{dt^2}(0)=u_2\in X^{\frac{n-3}{n}},\ldots,\ \dfrac{d^{n-1}u}{dt^{n-1}}(0)=u_{n-1}\in X, 
\end{equation}
that is,
\begin{equation}\label{Eq2}
\dfrac{d^iu}{dt^i}(0)=u_i\in X^{\frac{n-(i+1)}{n}},\quad i\in\{0,1,\ldots, n-1\},\ n\geqslant3,
\end{equation}
where $X$ be a separable Hilbert space and $A:D(A)\subset X\to X$ be an unbounded linear, closed, densely defined, self-adjoint and positive definite operator. Choose $\lambda_0>0$ such that $\mbox{Re}\sigma(A)>\lambda_0$, that is,   $\mbox{Re}\lambda>\lambda_0$ for all $\lambda\in\sigma(A)$, where $\sigma(A)$ denotes is the spectrum of $A$, and therefore, $A$ is a sectorial operator in the sense of Henry \cite[Definition 1.3.1]{H}. This allows us to define the fractional power $A^{\alpha}$ of order $0<\alpha<1$ according to Amann \cite{A} and Henry \cite{H}, as a closed linear operator, which can be given by
\begin{equation}\label{FracPower}
A^{\alpha} = \frac{\sin(\alpha \pi)}{\pi}\int_{0}^{\infty} \lambda^{\alpha-1} A(\lambda I+A)^{-1}d\lambda,
\end{equation}
see e.g. Hasse \cite{Has}, Henry \cite{H}, Krein \cite{Kr} and Mart\'inez and Sanz \cite{MS}.

Denote by $X^{\alpha}=D(A^{\alpha})$ for $0\leqslant\alpha\leqslant1$ (taking $A^0:=I$ on $X^0:=X$ when $\alpha=0$). Recall that $X^{\alpha}$ is dense in $X$  for all $0\leqslant\alpha\leqslant1$, for details see Amann \cite[Theorem 4.6.5]{A}. The fractional power space $X^\alpha$ endowed with the norm 
\[
\|\cdot\|_{X^\alpha}:=\|A^{\alpha} \cdot\|_X
\] 
is a Banach space. It is not difficult to show that $A^{\alpha}$ is the generator of a strongly continuous analytic
semigroup on $X$, that we will denote by $\{e^{-tA^{\alpha}}:  t\geqslant 0\}$, see Kre\v{\i}n \cite{Kr} and  Tanabe \cite{Ta} for any $\alpha\in[0,1]$. With this notation, we have $X^{-\alpha}=(X^\alpha)'$ for all $\alpha>0$, see Amann \cite{A}, Sobolevski\u{\i} \cite{S} and Triebel \cite{T} for the characterization of the negative scale. 

Fractional powers approach of operators for Cauchy problems associated with evolutionary equations has been studied in some situations, in the sense of existence, regularity and behavior asymptotic of solutions for theses problems, see e.g., Bezerra, Carvalho, Cholewa, and  Nascimento \cite{BCCN}, Bezerra, Carvalho, and  Nascimento \cite{BCN}, Bezerra and  Nascimento \cite{BN}, our previous paper Bezerra and Santos \cite{BS}, and references therein.

The problem \eqref{Eq1}-\eqref{Eq1ma} was studied by  Fattorini \cite{HOF}, and thanks to Fattorini this problem is well-posed if the following two conditions are satisfied: $(i)$  there is a dense space $D$ of $X$ such that for any $u_0,u_1,\ldots, u_{n-1}\in D$, there exists a unique classical solution in $X$; $(ii)$ the continuous dependence of the solution $u(t)$ on the initial data uniformly for $t$ in any given bounded interval holds. In \cite{HOF} it is proved that the Cauchy problem with $n\geqslant3$ is well-posed if and only if $A$ is a bounded linear operator on $X$.

We also recall that if $n=1$ then \eqref{Eq1}-\eqref{Eq1ma} is a linear parabolic problem, and it is well known that this Cauchy problem is well-posed (in the sense of Hadamard) via theory of analytic semigroups, see e.g. Amann \cite{A}, Cholewa, and T. D\l otko \cite{ChD}, Henry \cite{H} and Sobolevski\u{\i} \cite{S}. If $n=2$ then \eqref{Eq1}-\eqref{Eq1ma} is a linear hyperbolic problem, and it is well known that this Cauchy problem is well-posed (in the sense of Hadamard) via theory of strongly continuous semigroups; more precisely, via theory of strongly continuous (semi)groups of unitary operators, see e.g. Pazy \cite{P}. 

In \cite{BS} we study approximations for a class of linear evolution equations of third order in time governed by fractional powers of an operator.  We explicitly calculate the fractional powers of matrix-valued operators associated with evolution equations of third order in time, and we characterize the partial scale of the fractional power of order spaces associated with these operators.

In this paper we study the problem \eqref{Eq1}-\eqref{Eq1ma} under point of view of the theory of strongly continuous semigroups of bounded linear operators in $X$  and its amazing relationship with the  Chebyshev polynomials of the second kind, see e.g., Amann \cite{A}, Henry \cite{H}, Kre\v{\i}n \cite{Kr}, Pazy \cite{P}, Sobolevski\u{\i} \cite{S}, Tanabe \cite{Ta} and Triebel \cite{T}.

To our best knowledge, there is not fractional powers approach of operators for evolution equations of $n$-th order in time  with $n\geqslant4$. Here, we continue the analysis done in the previously cited works.  To better present our results in this paper, we introduce some notations and terminologies. 

We will rewrite \eqref{Eq1}-\eqref{Eq2} as a first order abstract system, we will calculate the fractional powers of order $\alpha\in(0,1)$ of the main part of the differential system and we are interested in finding for which values of  $\alpha\in [0,1]$ the fractional first order system is well-posed in some sense; namely, we will consider the phase space 
\[
Y=X^{\frac{n-1}{n}}\times X^{\frac{n-2}{n}}\times X^{\frac{n-3}{n}}\times\cdots\times X
\] 
which is a Banach space equipped with the norm given by
\[
\|\cdot\|_Y^2=\|\cdot\|^2_{X^{\frac{n-1}{n}}}+\|\cdot\|^2_{X^{\frac{n-2}{n}}}+\|\cdot\|^2_{X^{\frac{n-3}{n}}}+\cdots+\|\cdot\|^2_{X}.
\]

We can write the problem \eqref{Eq1}-\eqref{Eq2} as a Cauchy problem on $Y$, letting $v_1=u$, $v_2=\frac{du}{dt}$, $v_3=\frac{d^2u}{dt^2}$, \ldots, $v_n=\frac{d^{n-1}u}{dt^{n-1}}$,  
\[
{\bf u}=\left[\begin{smallmatrix} v_1\\ v_2 \\ v_3\\ \vdots \\ v_n \end{smallmatrix}\right]\quad\mbox{and}\quad {\bf u}_0=\left[\begin{smallmatrix} u\\ u_1 \\ u_2\\ \vdots \\ u_{n-1} \end{smallmatrix}\right],
\] 
and the initial value problem
\begin{equation}\label{mainprob}
\begin{cases}
\dfrac{d {\bf u}}{dt}+ \varLambda_n {\bf u} = 0,\ t>0,\\
{\bf u}(0)={\bf u}_0,
\end{cases}
\end{equation}
where the unbounded linear operator $\varLambda_n :D(\varLambda_n)\subset Y \to Y$ is defined by
\begin{equation}\label{LamOp1}
D(\varLambda_n)=X^1\times X^{\frac{n-1}{n}}\times X^{\frac{n-2}{n}}\times\cdots\times X^{\frac{1}{n}},
\end{equation}
equipped with the norm given by
\[
\|\cdot\|^2=\|\cdot\|^2_{X^1}+\|\cdot\|^2_{X^{\frac{n-1}{n}}}+\|\cdot\|^2_{X^{\frac{n-2}{n}}}+\cdots+\|\cdot\|^2_{X^{\frac{1}{n}}}.
\]
and
\begin{equation}\label{LamOp2}
\varLambda_n {\bf u}= \left[\begin{smallmatrix}
0 & -I & 0 & \cdots & 0 & 0\\0 & 0 & -I & \cdots & 0  & 0\\ 0 & 0 & 0 & \cdots & 0 & 0\\ \vdots & \vdots & \vdots & \ddots & \vdots & \vdots\\
0 & 0  & 0 & \cdots & 0 & -I\\
 A & 0 & 0 & \cdots & 0 & 0
\end{smallmatrix}\right] {\left[\begin{smallmatrix} v_1\\ v_2 \\ v_3\\ \vdots \\ v_{n-1}\\ v_n \end{smallmatrix}\right]}:= {\left[\begin{smallmatrix} -v_2\\ -v_3 \\ -v_4\\ \vdots \\ -v_n \\ Av_1 \end{smallmatrix}\right]},\ \forall {\bf u}={\left[\begin{smallmatrix} v_1\\ v_2 \\ v_3\\ \vdots \\ v_n \end{smallmatrix}\right]}\in D(\varLambda_n).
\end{equation}

From now on, we denote 
\[
Y^1=D(\varLambda_n)=X^1\times X^{\frac{n-1}{n}}\times X^{\frac{n-2}{n}}\times\cdots\times X^{\frac{1}{n}},
\]
equipped with the norm
\[
\|\cdot\|_{Y^1}^2=\|\cdot\|^2_{X^1}+\|\cdot\|^2_{X^{\frac{n-1}{n}}}+\|\cdot\|^2_{X^{\frac{n-2}{n}}}+\cdots+\|\cdot\|^2_{X^{\frac{1}{n}}}.
\]

Summaries of our main results and the structure of the paper are as follows. In Section \ref{S2} entitled ‘The fractional powers of the operators' we recall fundamental aspects of the theory of fractional powers of non-negative  (and not necessarily self-adjoint) operators, which have important links with partial differential equations.  Moreover, under Balakrishnan formula, see \cite{Ba}, we explicitly calculate the fractional powers of the operator $\varLambda_n$ for $0<\alpha<1$, and we study spectral properties of the fractional power operators $\varLambda_n^\alpha$, as well as, we characterize the partial scale of the fractional power spaces associated with these operators. Namely, we prove the following theorems.

\begin{theorem}\label{Theorem_1a}
Let $\varLambda_n$ be the unbounded linear operator  defined in \eqref{LamOp1}-\eqref{LamOp2}. Then $-\varLambda_n$ is not the infinitesimal generator of a strongly continuous semigroup on $Y$.
\end{theorem}

\begin{theorem}\label{lem:useful:properties}
Let $\varLambda_n$ be the unbounded linear operator  defined in \eqref{LamOp1}-\eqref{LamOp2}. Then we have all the following.
\begin{itemize}
\item[$i)$] $0\in \rho(\varLambda_n)$ and 
\[
\varLambda_n^{-1}=\left[\begin{smallmatrix}
0 & 0 & 0 & \cdots & 0 & A^{-1}\\
-I & 0 & 0 & \cdots & 0  & 0\\ 
0 & -I & 0 & \cdots & 0 & 0\\
 \vdots & \vdots & \vdots & \ddots & \vdots & \vdots\\
0 & 0  & 0 & \cdots & 0 & 0\\
 0 & 0 & 0 & \cdots & -I & 0
\end{smallmatrix}\right].
\]
Moreover, if $A$ has compact resolvent on $X$, then $\varLambda_n$ has compact resolvent on $Y$;
\item[$ii)$] Fractional powers $\varLambda_n^\alpha$ can be defined for $0<\alpha<1$ by
\begin{equation}\label{eq:to-compute-fractional-powers}
\varLambda_n^{\alpha}= \dfrac{\sin(\alpha \pi)}{\pi}\int_0^\infty\lambda^{\alpha-1}\varLambda_n(\lambda I+\varLambda_n)^{-1}d\lambda;
\end{equation}
\item[$iii)$] Given any  $0\leqslant\alpha\leqslant1$ we have the unbounded linear operator $\varLambda_n^{\alpha}:D(\varLambda_n^{\alpha})\subset Y\to Y$ defined by
\begin{equation}\label{asas324}
D(\varLambda_n^{\alpha})=X^{\frac{\alpha+n-1}{3}}\times X^{\frac{\alpha+n-2}{3}}\times X^{\frac{\alpha+n-3}{3}}\times\cdots\times X^{\frac{\alpha}{3}},
\end{equation}
and
\begin{equation}\label{pota}
\varLambda_n^{\alpha}= \left[\frac{(-1)^{i-j}}{n}U_{n-1}\left(cos\left(\frac{(\alpha+i-j)\pi}{n}\right)\right)A^{\frac{\alpha+i-j}{n}}\right]_{ij},
\end{equation}
where $U_n:\mathbb{C}\to\mathbb{C}$ is the $n$th degree Chebyshev polynomial of second kind defined by the recurrence relation
\begin{itemize}
\item[] $U_0(x)=1$,
\item[] $U_1(x)=2x$,
\item[] $U_2(x)=4x^2-1$,
\item[] $U_3(x)=8x^3-4x$,
\item[] $U_{n+1}(x)= 2xU_n(x)-U_{n-1}(x)$, 
\end{itemize}
for all $x\in\mathbb{C}$ and  $n\geqslant 3$. 
\end{itemize}
\end{theorem}

From now on we will assume that $A$ has compact resolvent on $X$.

\begin{theorem}\label{0intervalowp0}
Let $\varLambda_n$ be the unbounded linear operator  defined in \eqref{LamOp1}-\eqref{LamOp2}.  For each $0<\alpha\leqslant1$ the spectrum of $-\varLambda_n^{\alpha}$ is such that the point spectrum consisting of eigenvalues  
\begin{equation}\label{eq:neigenvalues} 
\bigcup_{k=0}^{\lfloor \frac{n-1}{2}\rfloor}\left(\left\{\mu_{j}^{\frac{\alpha}{n}}e^{ i \frac{\pi (n- (n-2k-1)\alpha)}{n}}:\ j\in\mathbb{N}\right\}\cup\left\{\mu_{j}^{\frac{\alpha}{n}}e^{ i \frac{\pi (n+ (n-2k-1)\alpha)}{n}}:\ j\in\mathbb{N}\right\}\right)
\end{equation}
where $\{\mu_j\}_{j\in\mathbb{N}}$ denotes the ordered sequence of eigenvalues of $A$ including their multiplicity and $\lfloor x\rfloor := \max \{z\in \mathbb{Z}\ |\  z\leq x\}$.
\end{theorem}

 In Section \ref{Fractional differential equations} entitled ‘Fractional differential equations' we study the fractional  differential equations of $n$th order in time governed by the fractional power operators $\varLambda_n^\alpha$ for $0<\alpha<1$; namely, we consider the initial value problem  in $Y$
 \begin{equation}\label{mainprobalpha}
\begin{cases}
\dfrac{d {\bf u}^\alpha}{dt}+ \varLambda_n^\alpha {\bf u}^\alpha =0,\ t>0,\\
{\bf u}^\alpha(0)={\bf u}^\alpha_0,
\end{cases}
\end{equation}
where $0<\alpha<1$ and $ \varLambda_n^\alpha$ is given by \eqref{eq:to-compute-fractional-powers}-\eqref{asas324}.

Finally, in Section \ref{Applications} entitled ‘Applications' we treated on applications from our results to partial differential equations.

\section{The fractional powers of the operators}\label{S2}

Let $\varLambda_n :D(\varLambda_n)\subset Y \to Y$ be the unbounded linear operator  defined in \eqref{LamOp1}-\eqref{LamOp2}. Before we proceed, let us recall the definition of operator of positive type $K\geqslant  1$. These are the operators that one can define the fractional power, for more details see Pazy  \cite[Section 2.2.6]{P} and Amann  \cite[Section 3.4.6]{A}.

\subsection{Spectral properties of the operator $\varLambda_n$}

In this subsection, we study spectral properties of the operator $\varLambda_n$.

\begin{definition}\label{positive type}
Let $E$ be a Banach space over a field $\mathbb{K}$ ($\mathbb{K}=\mathbb{R}$ or $\mathbb{K}=\mathbb{C}$), and let $T: D(T) \subset E \to E$ be a linear operator. We say that $T$ is of positive type $K\geqslant  1$ if it is closed, densely defined, $[0,\infty)\subset \rho(-T)$ and
\begin{equation}\label{estres}
\|(\lambda I+T)^{-1}\|_{\mathcal{L}(E)}\leqslant \frac{K}{1+\lambda}, \ \mbox{for all}\ \lambda\geqslant 0.
\end{equation}
Here, $\mathcal{L}(E)$ denotes the space of linear operators defined in $E$ into self endowed with norm
\[
\|S\|_{\mathcal{L}(E)}:=\sup_{x\in E,\ x\neq0}\dfrac{\|Sx\|_{E}}{\|x\|_{E}},\ \forall S\in \mathcal{L}(E).
\] 
\end{definition}

\begin{remark}\label{generator-positive}
Let $T:D(T)\subset E\to E$ be a linear operator  on some Banach space $E$. It is well known that, if the operator $-T$ is the generator of a strongly continuous semigroup which decays exponentially  then $T$ is of positive type, see Amann \cite[Page 156]{A}. It is not true, however, that if an  operator $T$ is of positive type then $-T$ generates a strongly continuous semigroup. It is the case of the operator in \eqref{LamOp2} as we shall see later.
\end{remark}

\begin{lemma}\label{lsetornores}[Pazy  \cite[Section 2.2.6]{P} and Amann  \cite[Section 3.4.6]{A}]
If $T: D(T) \subset E \to E$ is a positive operator of type $K\geq 1$, then 
\begin{equation}\label{setornores}
S(K):=\left\{\lambda\in\mathbb{C}:|arg\ \lambda|\leqslant \arcsin \frac{1}{2K}\right\} \cup\left\{|\lambda|\leq\frac{1}{2K}\right\}\subset\rho(-T)
\end{equation} 
and
\begin{equation}
(1+|\lambda|)\|(\lambda I+T)^{-1}\|_{\mathcal{L}(E)}\leqslant 2K+1,\ \mbox{for all}\ \lambda\in S(K).
\end{equation}
\end{lemma}

\begin{corollary}\label{setorgener}Pazy  \cite[Section 2.2.6]{P} and Amann  \cite[Section 3.4.6]{A}
Suppose that $T:D(T)\subset E\to E$ is a linear operator of positive type and there exists $\theta\in (0,\pi)$ such that \eqref{estres} is satisfied for $\lambda \in \mathbb{C}$ with $|arg\ \lambda|\leq \theta$. If $\alpha\in(0,1)$ satisfies $\alpha<\pi/2(\pi-\theta)$ then $-T^{\alpha}$ generates a strongly continuous analytic semigroup on $E$.
\end{corollary}

\begin{lemma}\label{alfaanalit}
Assume that $T:D(T)\subset E\to E$ is a linear operator of positive type $K\geqslant 1$ and $-T$ generates a $C_0$-semigroup of contractions on $E$. Then $-T^{\alpha}$ generates a strongly continuous analytic semigroup on $E$ for $0<\alpha<1$.
\end{lemma}

\begin{proof}
Fix $\alpha\in\mathbb{R}$ with $0<\alpha<1$ and choose $\theta\in(0,\pi)$ such that 
\[
\pi-\frac{\pi}{2\alpha}<\theta<\frac{\pi}{2}.
\]
For such $\theta$ we have
\[
\alpha<\frac{\pi}{2(\pi-\theta)}<1.
\]

Thus, from Corollary \ref{setorgener}, it is sufficient to show that \eqref{estres} is satisfied for any $\lambda \in S_{\theta}:=\{\lambda\in\mathbb{C}:|arg\ \lambda|\leq \theta\}$. By Lemma \ref{lsetornores}, we only need to consider the case $\lambda\in S_{\theta}\setminus S(K)$. It follows from Hille-Yosida Theorem that
\[
\begin{split}
(1+|\lambda|)\|(\lambda I+T)^{-1}\|_{\mathcal{L}(E)}&\leqslant \frac{1+|\lambda|}{Re\ \lambda}\\
&\leqslant\frac{1}{Re\ \lambda}+\frac{1}{\cos\theta}\\
&\leqslant\frac{\cos^2\theta+2K}{2K\cos\theta}
\end{split}
\]
for any $\lambda\in S_{\theta}\setminus S(K)$.\qed
\end{proof}

\medskip

From this, we will show  that it is possible to calculate explicitly the fractional power $\varLambda_n^\alpha$ of the operator $\varLambda_n$ for  $0<\alpha<1$, and with this,  we will  consider the fractional approximations of \eqref{mainprob} given by
\begin{equation}\label{fracmainprob}
\begin{cases}
\dfrac{d {\bf u}^\alpha}{dt}+ \varLambda_n^\alpha {\bf u}^\alpha = 0,\ t>0,\ 0<\alpha<1,\\
{\bf u}^\alpha(0)={\bf u}_0^\alpha.
\end{cases}
\end{equation} 
Here, $\varLambda_n^\alpha:D(\varLambda_n^\alpha)\subset Y\to Y$ denotes the fractional power operator of $\varLambda_n$ to be defined by $\varLambda_n^\alpha=(\varLambda_n^{-\alpha})^{-1}$, where $\varLambda_n^{-\alpha}$ is given by the formula in \eqref{FracPower} with domain $D(\varLambda_n^\alpha)$ characterized by complex interpolation methods, see e.g. Amann \cite{A} and Cholewa and D\l otko \cite{ChD}.

\begin{lemma}\label{resoln} 
The resolvent set of $-\varLambda_n$ is given by
\begin{equation}\label{resolvlambda} 
\rho(-\varLambda_n)=\{\lambda \in \mathbb{C}: \lambda^n \in \rho(-A)\}.
\end{equation}
\end{lemma}

\proof
Suppose that $\lambda\in\mathbb{C}$ is such that $\lambda^n\in\rho(-A)$. We claim that $\lambda\in\rho(-\varLambda_n)$. Indeed, since $-\varLambda_n$ is a closed operator, we only need to show that 
$$\lambda I+\varLambda_n:D(\varLambda_n)\subset Y\to Y$$ 
is bijective; namely,
\begin{equation}\label{dasd4}
\lambda I+\varLambda_n= \left[\begin{smallmatrix}
\lambda I & -I & 0 & \cdots & 0 & 0\\0 & \lambda I & -I & \cdots & 0  & 0\\ 0 & 0 & \lambda I & \cdots & 0 & 0\\ \vdots & \vdots & \vdots & \ddots & \vdots & \vdots\\
0 & 0  & 0 & \cdots & \lambda I & -I\\
 A & 0 & 0 & \cdots & 0 & \lambda I
\end{smallmatrix}\right],\ \mbox{for all}\ \lambda\in\rho(-\varLambda_n),
\end{equation}

For injectivity consider ${\bf u}=\left[\begin{smallmatrix}
u_1\\
u_2\\
\vdots\\
u_n
\end{smallmatrix}\right] \in D(\varLambda_n)$ and $(\lambda I+\varLambda_n){\bf u}=0$, then

\begin{equation}\label{sistinj}
\begin{cases}
\lambda u_i-u_{i+1}= 0,\ \text{\ for\  }1\leq i\leq n-1\\
Au_1 + \lambda u_n = 0.
\end{cases}
\end{equation}
From \eqref{sistinj} we have
\begin{equation}
(\lambda^n I+A)u_1=0.
\end{equation}  
Since $\lambda^n\in \rho(-A)$, we conclude that $u_1=0$ and consequently $\textbf{u}=0$. 

For surjectivity given    
$\varphi=\left[\begin{smallmatrix}
\varphi_1\\
\varphi_2\\
\vdots\\
\varphi_n
\end{smallmatrix}\right]\in Y$  we take ${\bf u}=\left[\begin{smallmatrix}
u_1\\
u_2\\
\vdots\\
u_n
\end{smallmatrix}\right]$ with ${\bf u}=(\lambda I+\varLambda_n)^{-1}\varphi$; namely, with an abuse of notation in the matrix representation of the bounded linear operator\linebreak $(\lambda I+\varLambda_n)^{-1}$, we can write
\begin{equation}\label{dasd4odd}
(\lambda I+\varLambda_n)^{-1}=(\lambda^{n}I+A))^{-1} \cdot
 \left[\begin{smallmatrix}
\lambda^{n-1} I & \lambda^{n-2} I   & \lambda^{n-3} I & \cdots & \lambda I   & I \\ \\ 
-A  &  \lambda^{n-1} I  & \lambda^{n-2}  I & \cdots & \lambda^2 I & \lambda I \\ \\ 
-\lambda A & -A  & \lambda^{n-1} I & \cdots & \lambda^3 I & \lambda^2 I \\ \\ 
\vdots & \vdots & \vdots & \ddots & \vdots & \vdots\\ \\ 
-\lambda^{n-3}A  & -\lambda^{n-4}A   & -\lambda^{n-5}A  & \cdots & \lambda^{n-1}  I& \lambda^{n-2} I  \\ \\ 
-\lambda^{n-2}A & -\lambda^{n-3}A   &  -\lambda^{n-4}A  & \cdots & -A & \lambda^{n-1} I \\ \\ 
\end{smallmatrix}\right]
\end{equation}
for all $\lambda\in\rho(-\varLambda_n)$. 

Observe that the matrix representation of the bounded linear operator $(\lambda I+\varLambda_n)^{-1}$ is a Toeplitz matrix; namely 
\[
(\lambda I+\varLambda_n)^{-1}=[A_{ij}],
\]
and if the $i,j$ element of $(\lambda I+\varLambda_n)^{-1}$ is denoted $A_{ij}$, then we have
\[
A_{ij}=A_{i+1,j+1}=a_{i-j}.
\]

In other words,
\begin{equation}\label{asad45tg}
u_1=(\lambda^{n}I+A))^{-1}  \left(\sum_{k=1}^n\lambda^{n-k}\varphi_k\right) 
\end{equation}
and
\begin{equation}\label{asad45tg2}
u_i=(\lambda^{n}I+A))^{-1}\left[ A\left(\sum_{k=1}^{i-2}\lambda^k\varphi_k\right)   +A\varphi_{i-1}+ \sum_{k=i}^n\lambda^{k-1}\varphi_k\right]
\end{equation}
for $2\leqslant i\leqslant n$.

Note that, for $1\leq i\leq n$, $u_i$ is well defined since $\lambda^n\in\rho(-A)$. Moreover, $u_i\in D(A^{\frac{n-i+1}{n}})$ for $1\leqslant i\leqslant n$ . Then we have ${\bf u}\in D(\varLambda_n)$ and 
\[
(\lambda I+\varLambda_n){\bf u}=\varphi. 
\]
Now suppose that $\lambda\in \rho(-\varLambda_n)$. If $u_1\in D(A)$ is such that $(\lambda^n I+A)u_1=0$, taking ${\bf u}=\left[\begin{smallmatrix}
u\\
\lambda u_1\\
\vdots\\
\lambda^{n-1}u_1
\end{smallmatrix}\right]\in D(\varLambda_n)$ we have
\begin{equation}
(\lambda I+\varLambda_n){\bf u} = 0.
\end{equation}
Since $\lambda\in \rho(-\varLambda_n)$, it follows that ${\bf u}=0$ and consequently $u_1=0$, which proves the injectivity of $\lambda^n I+A$. Given $f\in X$, consider $\varphi=\left[\begin{smallmatrix}
0\\
0\\
f
\end{smallmatrix}\right] \in Y$. By the surjectivity of $\lambda I + \varLambda_n$ there exists ${\bf u}=\left[\begin{smallmatrix}
u_1\\
u_2\\
\vdots\\
u_n
\end{smallmatrix}\right] \in D(\varLambda_n)$ such that
\begin{equation}
(\lambda I + \varLambda_n){\bf u}=\varphi
\end{equation}
which gives 
\[
(\lambda^n I+A)u_1=f
\]
and the proof is complete. 
\qed

\begin{remark}
The inversion of a Toeplitz matrix is usually not a Toeplitz matrix. But, in our case, we have $(\lambda I+\varLambda_n)$ and $(\lambda I+\varLambda_n)^{-1}$ Toeplitz matrices for any $\lambda\in\rho(-\varLambda_n)$.
\end{remark}

\subsection{Proof of Theorem \ref{Theorem_1a}: Ill-posed problems}

If $-\varLambda_n$ generates a strongly continuous semigroup $\{e^{-\varLambda_n t}:t\geqslant0\}$ on $Y$, it follows from  Pazy \cite[Theorem 1.2.2]{P} that there exist  constants $\omega\geq 0$ and $M\geq 1$ such that

\begin{equation}
\|e^{-\varLambda_n t}\|_{\mathcal{L}(Y)}\leq Me^{\omega t}\ \ \text{\ for \ } 0\leq t< \infty.
\end{equation}
Moreover, from Pazy \cite[Remark 1.5.4]{P} we have

\begin{equation}\label{semiplan}
\{\lambda \in \mathbb{C}: Re\lambda > \omega\} \subset \rho(-\varLambda_n) 
\end{equation}
where $\rho(-\varLambda_n)$ denotes the resolvent set of the operator $-\varLambda_n$.

Let ${\bf u}={\left[\begin{smallmatrix} v_1\\ v_2 \\ v_3\\ \vdots \\ v_n \end{smallmatrix}\right]}$ be a nontrivial element of $D(\varLambda_n)$. We shall consider the eigenvalue problem for the operator  $-\varLambda_n$ 
\[
-\varLambda_n{\bf u} = \lambda{\bf u}.
\]
A straightforward calculation implies 
\[
\sigma_p(-\varLambda_n)=\{\lambda \in \mathbb{C}: \lambda^n \in \sigma_p(-A)\},
\]
where $\sigma_p(-\varLambda_n)$ and $\sigma_p(-A)$ denotes the point spectrum set of $-\varLambda_n$ and $-A$, respectively.
Since $\sigma_p(-A)=\{-\mu_j: j\in \mathbb{N}\}$ with $\mu_j\in \sigma_p(A)$ for each $j\in \mathbb{N}$ and $\mu_j\to \infty$ as $j\to \infty$, we conclude that
\[
\sigma_p(-\varLambda_n)\cap \{\lambda \in \mathbb{C}:Re \lambda>\omega\}\neq \emptyset
\]
This contradicts the equation \eqref{semiplan} and therefore $-\varLambda_n$ can not be the infinitesimal generator of a strongly continuous semigroup on $Y$.
\qed

\begin{remark}
Thanks to Proof of  Theorem \ref{Theorem_1a} we can see that the location of the eigenvalues of $\varLambda_n$ is a reunion of the $n$ semi-lines $\{r^{\frac{1}{n}}e^{\frac{i\pi\varrho}{n}};\ r>0\ \mbox{and}\ \varrho\in\{1,3,5,\ldots, 2n-1\}\}$; namely,  for each  the spectrum of $\varLambda_n$ is such that the point spectrum consisting of eigenvalues  
\begin{equation}\label{eq:neigenvaluesLimit} 
\bigcup_{k=0}^{\lfloor \frac{n-1}{2}\rfloor}\left(\left\{\mu_{j}^{\frac{1}{n}}e^{ \frac{i\pi (-n+2k+1)}{n}}:\ j\in\mathbb{N}\right\}\cup\left\{\mu_{j}^{\frac{1}{n}}e^{ \frac{i\pi (n-2k-1)}{n}}:\ j\in\mathbb{N}\right\}\right)
\end{equation}
where $\{\mu_j\}_{j\in\mathbb{N}}$ denotes the ordered sequence of eigenvalues of $A$ including their multiplicity and $\lfloor x\rfloor := \max \{z\in \mathbb{Z}\ |\  z\leq x\}$. 

Note that if $n$ is odd, then $\frac{n-1}{2}$ is a positive integer number $(\lfloor \frac{n-1}{2}\rfloor=\frac{n-1}{2})$, and if $k_0=\frac{n-1}{2}$ in \eqref{eq:neigenvaluesLimit}, then $e^{ \frac{i\pi (-n+2k_0+1)}{n}}=e^{ \frac{i\pi (n-2k_0-1)}{n}}$.

Note that if $n$ is even, then $\frac{n-1}{2}$ is not an integer number and $\lfloor \frac{n-1}{2}\rfloor=\frac{n-2}{2}$. In this case $e^{ \frac{i\pi (-n+2k+1)}{n}}\neq e^{ \frac{i\pi (n-2k-1)}{n}}$ for any $k\in\{0,1,2,\ldots,\frac{n-2}{2}\}$.

For example, see the figures below for $n=2,3,4$ and $5$.
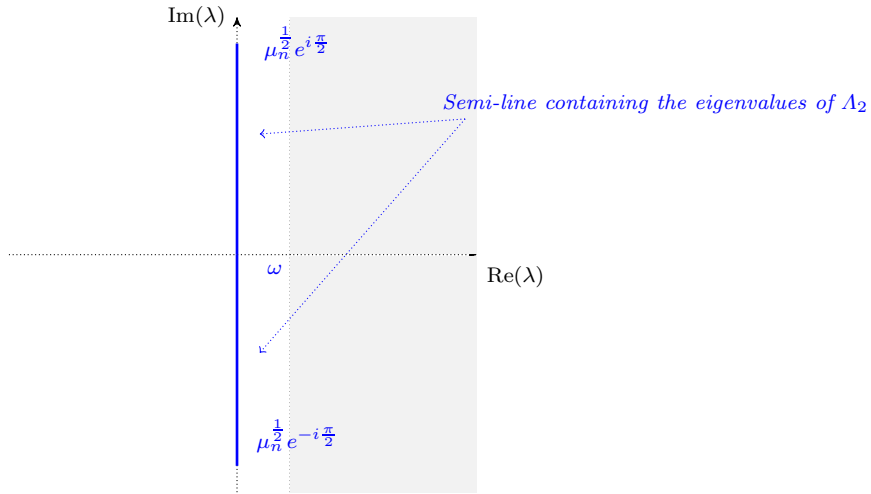
\begin{figure}[H] 
\begin{center}
\begin{tikzpicture}
\draw[-stealth', densely dotted] (-3.0,0) -- (3.15,0) node[below] {\ \ \ \ \ \ \ $\scriptstyle {\rm Re} (\lambda)$};
\draw[-stealth', densely dotted ] (0,-3.15) -- (0,3.15) node[left] {\color{black}$\scriptstyle{\rm Im} (\lambda)$};
\draw[densely dotted ] (0.7,-3.15) -- (0.7,3.15);
\fill[gray!10!] (0.7,-3.15) rectangle (3.15,-0.02);
\fill[gray!10!] (0.7,0.02) rectangle (3.15,3.15);
\node at (5.5,2) {{\tiny\color{blue} Semi-line containing the eigenvalues of $\varLambda_2$}};
\draw[color=blue,->, densely dotted] (3,1.8) -- (0.3,1.6);
\draw[color=blue,->, densely dotted] (3,1.8) -- (0.3,-1.3);
\draw[color=blue, line width=1pt] (0,2.8) -- (0,0);
\draw[color=blue, line width=1pt] (0,-2.8) -- (0,0);
\node at (0.5,-0.2) {\color{blue}{\tiny $\omega$}};
\node at (0.8,2.8) {\color{blue}{\tiny $\mu_n^{\frac{1}{2}}e^{i\frac{\pi}{2}}$}};
\node at (0.8,-2.4) {\color{blue}{\tiny $\mu_n^{\frac{1}{2}}e^{-i\frac{\pi}{2}}$}};
\end{tikzpicture}
\end{center}
\caption{Location of the eigenvalues of $\varLambda_2$ with $n=2$}\label{fig0001}
\end{figure}
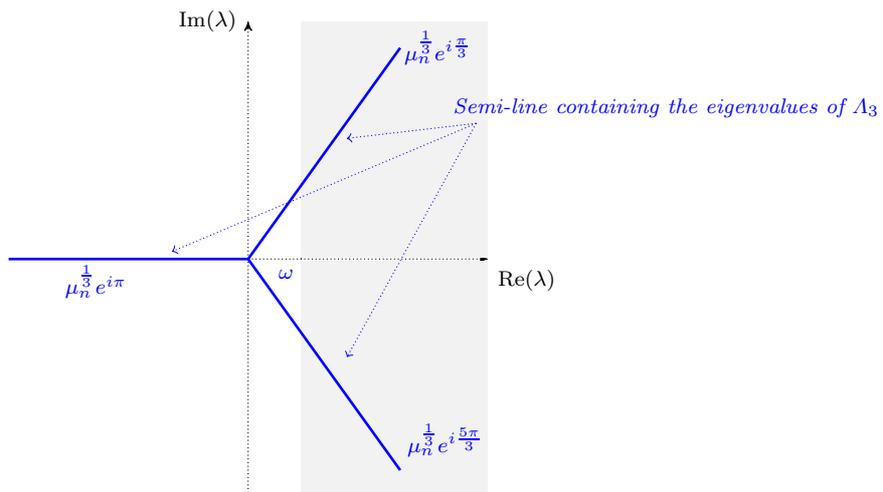
\begin{figure}[H]
\begin{center}
\begin{tikzpicture}
\draw[-stealth', densely dotted] (0,0) -- (3.15,0) node[below] {\ \ \ \ \ \ \ $\scriptstyle {\rm Re} (\lambda)$};
\draw[-stealth', densely dotted ] (0,-3.15) -- (0,3.15) node[left] {\color{black}$\scriptstyle{\rm Im} (\lambda)$};
\draw[densely dotted ] (0.7,-3.15) -- (0.7,3.15);
\fill[gray!10!] (0.7,-3.15) rectangle (3.15,-0.02);
\fill[gray!10!] (0.7,0.02) rectangle (3.15,3.15);
\node at (5.5,2) {{\tiny\color{blue} Semi-line containing the eigenvalues of $\varLambda_3$}};
\draw[color=blue,->, densely dotted] (3,1.8) -- (1.3,1.6);
\draw[color=blue,->, densely dotted] (3,1.8) -- (1.3,-1.3);
\draw[color=blue,->, densely dotted] (3,1.8) -- (-1,0.1);
\draw[color=blue, line width=1pt] (-3.15,0) -- (0,0);
\draw[color=blue, line width=1pt] (2,2.8) -- (0,0);
\draw[color=blue, line width=1pt] (2,-2.8) -- (0,0);
\node at (-2.0,-0.3) {\color{blue}{\tiny $\mu_n^{\frac{1}{3}}e^{i\pi}$}};
\node at (0.5,-0.2) {\color{blue}{\tiny $\omega$}};
\node at (2.5,2.8) {\color{blue}{\tiny $\mu_n^{\frac{1}{3}}e^{i\frac{\pi}{3}}$}};
\node at (2.6,-2.4) {\color{blue}{\tiny $\mu_n^{\frac{1}{3}}e^{i\frac{5\pi}{3}}$}};
\end{tikzpicture}
\end{center}
\caption{Location of the eigenvalues of $\varLambda_3$ with $n=3$}\label{fig01}
\end{figure}
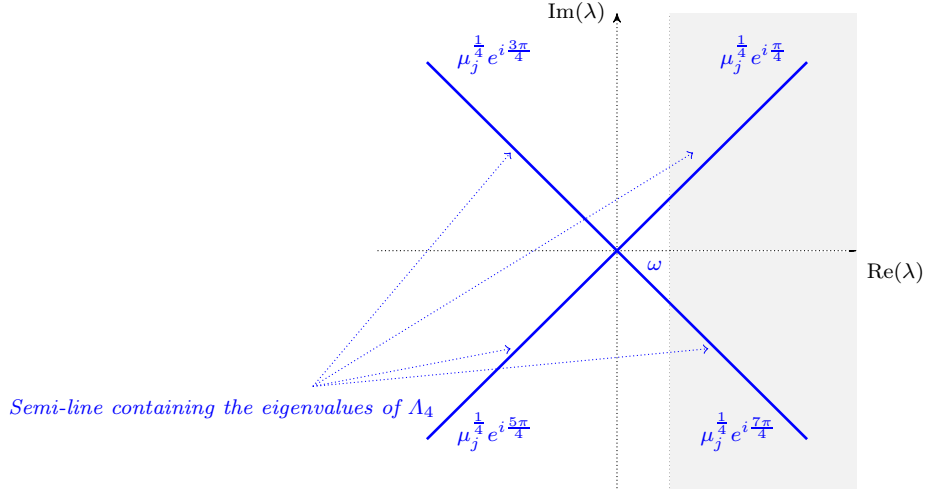
\begin{figure}[H]
\begin{center}
\begin{tikzpicture}
\draw[-stealth', densely dotted] (-3.15,0) -- (3.15,0) node[below] {\ \ \ \ \ \ \ $\scriptstyle {\rm Re} (\lambda)$};
\draw[-stealth', densely dotted ] (0,-3.15) -- (0,3.15) node[left] {\color{black}$\scriptstyle{\rm Im} (\lambda)$};
\draw[densely dotted ] (0.7,-3.15) -- (0.7,3.15);
\fill[gray!10!] (0.7,-3.15) rectangle (3.15,-0.02);
\fill[gray!10!] (0.7,0.02) rectangle (3.15,3.15);
\node at (-5.2,-2.1) {{\tiny\color{blue} Semi-line containing the eigenvalues of $\varLambda_4$}};
\draw[color=blue,->, densely dotted] (-4,-1.8) -- (-1.4,-1.3);
\draw[color=blue,->, densely dotted] (-4,-1.8) -- (1.2,-1.3);
\draw[color=blue,->, densely dotted] (-4,-1.8) -- (-1.4,1.3);
\draw[color=blue,->, densely dotted] (-4,-1.8) -- (1,1.3);
\draw[color=blue, line width=1pt] (-2.5,2.5) -- (0,0);
\draw[color=blue, line width=1pt] (2.5,2.5) -- (0,0);
\draw[color=blue, line width=1pt] (-2.5,-2.5) -- (0,0);
\draw[color=blue, line width=1pt] (2.5,-2.5) -- (0,0);
\node at (0.5,-0.2) {\color{blue}{\tiny $\omega$}};
\node at (-1.6,2.6) {\color{blue}{\tiny $\mu_j^{\frac14}e^{i\frac{3\pi}{4}}$}};
\node at (1.8,2.6) {\color{blue}{\tiny $\mu_j^{\frac14}e^{i\frac{\pi}{4}}$}};
\node at (-1.6,-2.4) {\color{blue}{\tiny $\mu_j^{\frac14}e^{i\frac{5\pi}{4}}$}};
\node at (1.6,-2.4) {\color{blue}{\tiny $\mu_j^{\frac14}e^{i\frac{7\pi}{4}}$}};
\end{tikzpicture}
\end{center}
\caption{Location of the eigenvalues of $\varLambda_4$ with $n=4$}\label{fig04}
\end{figure}
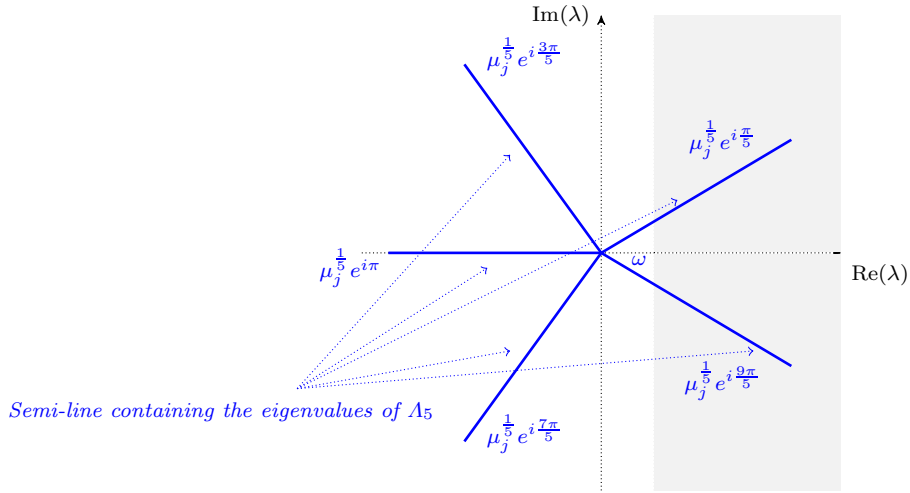
\begin{figure}[H]
\begin{center}
\begin{tikzpicture}
\draw[-stealth', densely dotted] (-3.15,0) -- (3.15,0) node[below] {\ \ \ \ \ \ \ $\scriptstyle {\rm Re} (\lambda)$};
\draw[-stealth', densely dotted ] (0,-3.15) -- (0,3.15) node[left] {\color{black}$\scriptstyle{\rm Im} (\lambda)$};
\draw[densely dotted ] (0.7,-3.15) -- (0.7,3.15);
\fill[gray!10!] (0.7,-3.15) rectangle (3.15,-0.02);
\fill[gray!10!] (0.7,0.02) rectangle (3.15,3.15);
\node at (-5,-2.1) {{\tiny\color{blue} Semi-line containing  the eigenvalues of $\varLambda_5$}};
\draw[color=blue,->, densely dotted] (-4,-1.8) -- (-1.2,-1.3);
\draw[color=blue,->, densely dotted] (-4,-1.8) -- (2,-1.3);
\draw[color=blue,->, densely dotted] (-4,-1.8) -- (-1.2,1.3);
\draw[color=blue,->, densely dotted] (-4,-1.8) -- (-1.5,-0.2);
\draw[color=blue,->, densely dotted] (-4,-1.8) -- (1,0.7);
\draw[color=blue, line width=1pt] (-1.8,2.5) -- (0,0);
\draw[color=blue, line width=1pt] (2.5,1.5) -- (0,0);
\draw[color=blue, line width=1pt] (-2.8,0) -- (0,0);
\draw[color=blue, line width=1pt] (-1.8,-2.5) -- (0,0);
\draw[color=blue, line width=1pt] (2.5,-1.5) -- (0,0);
\node at (0.5,-0.1) {\color{blue}{\tiny $\omega$}};
\node at (-1,2.6) {\color{blue}{\tiny $\mu_j^{\frac15}e^{i\frac{3\pi}{5}}$}};
\node at (1.6,1.5) {\color{blue}{\tiny $\mu_j^{\frac15}e^{i\frac{\pi}{5}}$}};
\node at (-3.3,-0.2) {\color{blue}{\tiny $\mu_j^{\frac15}e^{i\pi}$}};
\node at (-1,-2.4) {\color{blue}{\tiny $\mu_j^{\frac15}e^{i\frac{7\pi}{5}}$}};
\node at (1.6,-1.7) {\color{blue}{\tiny $\mu_j^{\frac15}e^{i\frac{9\pi}{5}}$}};
\end{tikzpicture}
\end{center}
\caption{Location of the eigenvalues of $\varLambda_5$ with $n=5$}\label{fig06}
\end{figure}
\end{remark}

\begin{lemma}\label{lemma1}
If $n$ is even, then the unbounded linear operator $-\varLambda_n$ with $\varLambda_n:D(\varLambda_n)\subset Y\to Y$ defined in \eqref{LamOp1}-\eqref{LamOp2} is  the infinitesimal generator of a strongly continuous group on $Y$.
\end{lemma} 

\proof
If $-\varLambda_n$ generates a strongly continuous semigroup $\{e^{-\varLambda_n t}:t\geqslant0\}$ on $Y$, it follows from  Pazy \cite[Theorem 1.2.2]{P} that there exist  constants $\omega\geq 0$ and $M\geq 1$ such that

\begin{equation}
\|e^{-\varLambda_n t}\|_{\mathcal{L}(Y)}\leq Me^{\omega t}\ \ \text{\ for \ } 0\leq t< \infty.
\end{equation}
Moreover, from Pazy \cite[Remark 1.5.4]{P} we have

\begin{equation}\label{semiplan}
\{\lambda \in \mathbb{C}: Re\lambda > \omega\} \subset \rho(-\varLambda_n) 
\end{equation}
where $\rho(-\varLambda_n)$ denotes the resolvent set of the operator $-\varLambda_n$.

Let 
\[
{\bf u}={\left[\begin{smallmatrix} v_1\\ v_2 \\ v_3\\ \vdots \\ v_n \end{smallmatrix}\right]}
\] 
be a nontrivial element of $D(\varLambda_n)$. We shall consider the eigenvalue problem for the operator  $-\varLambda_n$ 
\[
-\varLambda_n{\bf u} = \lambda{\bf u}.
\]
A straightforward calculation implies 
\[
\sigma_p(-\varLambda_n)=\{\lambda \in \mathbb{C}: \lambda^n \in \sigma_p(-A)\}.
\]
Where $\sigma_p(-\varLambda_n)$ and $\sigma_p(-A)$ denote the point spectrum set of $-\varLambda_n$ and $-A$, respectively.
Since $\sigma_p(-A)=\{-\mu_j: j\in \mathbb{N}\}$ with $\mu_j\in \sigma_p(A)$ for each $j\in \mathbb{N}$ and $\mu_j\to \infty$ as $j\to \infty$, we conclude that
\[
\sigma_p(-\varLambda_n)\cap \{\lambda \in \mathbb{C}:Re \lambda>\omega\}\neq \emptyset
\]
This contradicts the equation \eqref{semiplan} and therefore $-\varLambda_n$ can not be the infinitesimal generator of a strongly continuous semigroup on $Y$.
\qed

\begin{remark}
We note that if $n\geqslant3$, then $-\varLambda_n$ is not a dissipative operator on $Y$, according to Pazy \cite[Definition 4.1, Chapter 1]{P}. Indeed, if $u$ be a non-trivial element in $X^{1}$ and 
\[
\mathbf{u}=\left[\begin{smallmatrix}
u\\
0\\
0\\
\vdots\\
0\\
-u 
\end{smallmatrix}\right]
\] 
then 
\[
\left<-\varLambda_n \mathbf{u},\mathbf{u}\right>_Y=\left<\left[\begin{smallmatrix} 0\\
0\\
0\\
\vdots\\
-u\\
-Au 
\end{smallmatrix}\right],\left[\begin{smallmatrix}
u\\
0\\
0\\
\vdots\\
0\\
-u 
\end{smallmatrix}\right]\right>_Y=\langle Au,u\rangle_X=\|u\|^2_{X^{\frac12}}>0.
\]
Explicitly, this means that $-\varLambda_n$ is not an infinitesimal generator of a strongly continuous semigroup of contractions on $Y$. Nevertheless, the statement in Lemma \ref{lemma1} is more precise because it says that $-\varLambda_n$ cannot be the infinitesimal generator of a strongly continuous semigroup of any type on $Y$.
\end{remark}

\begin{theorem}(Moment Inequality)
Let $T:D(T)\subset E\to E$ be a linear operator of positive type. If $\alpha\in[0,\infty)$, then there exists a constant $K>0$ such that 
\begin{equation}
\|x\|_{E^{\alpha}}\leq \frac{K}{2}(\rho^{\alpha}\|x\|_E+\rho^{\alpha-1}\|x\|_{E^1}),
\end{equation}
\begin{equation}\label{moment}
\|x\|_{E^{\alpha}}\leq K\|x\|_{E}^{1-\alpha}\|x\|_{E^1}^{\alpha}
\end{equation}
\end{theorem}

\begin{lemma}\label{estpotb}
Let $T:D(T)\subset E\to E$ be a linear operator of positive type. If $\alpha\in[0,1]$ and $\lambda\geq 0$, then $T^{\alpha}(\lambda I+T)^{-1}\in \mathcal{L}(E)$ and 
\begin{equation}
\|T^{\alpha}(\lambda I+T)^{-1}\|_{\mathcal{L}(E)}\leq \frac{K}{(1+\lambda)^{1-\alpha}}
\end{equation}
for some $K\geq 1$.
\end{lemma}

\proof
Here, $K$ will denote a positive constant, not necessarily the same one.
We first observe that
\begin{equation}\nonumber
T(\lambda I + T)^{-1} = I-\lambda(\lambda I + T)^{-1} .
\end{equation}
This and the fact that $T$ is of positive type give
\begin{equation}\label{opresbound}
\begin{split}
\|T(\lambda I + T)^{-1}\|_{\mathcal{L}(E)}&\leqslant 1 + \lambda\|(\lambda I + T)^{-1}\|_{\mathcal{L}(E)}\\
&\leqslant 1+ K.
\end{split}
\end{equation}
Now, for $x\in E$, from the inequality \eqref{moment} we have 
\[
\begin{split}
\|T^{\alpha}(\lambda I+T)^{-1}x\|_E &\leqslant K \|(\lambda I+T)^{-1}x\|_E^{1-\alpha}\|T(\lambda I+T)^{-1}x\|_E^{\alpha}\\
&\leqslant \frac{K}{(1+\lambda)^{1-\alpha}}
\end{split}
\]
for some $K\geq 1$. In the last inequality we use the fact that $T$ is of positive type and that $\|T(\lambda I+T)^{-1}x\|_E^{\alpha}$ is bounded by \eqref{opresbound}.
\qed

\begin{lemma}\label{corpos}
The unbounded linear operator $\varLambda_n$ defined in \eqref{LamOp1}-\eqref{LamOp2} is of positive type $K\geqslant 1$.
\end{lemma}

\proof
Firstly, we show that the operator $\varLambda_n$ is closed. Indeed, if ${\bf u}_j={\left[\begin{smallmatrix} u_{1,j}\\ u_{2,j} \\ u_{3,j}\\ \vdots \\ u_{n,j} \end{smallmatrix}\right]}\in D(\varLambda_n)$ with ${\bf u}_j\to {\bf u}={\left[\begin{smallmatrix} u_1\\ u_2 \\ u_3\\ \vdots \\ u_n \end{smallmatrix}\right]}$ in $Y$ as $j\to\infty$, and $\varLambda_n {\bf u}_j\to \varphi$ in $Y$ as $j\to\infty$, where $\varphi={\left[\begin{smallmatrix} \varphi_1\\ \varphi_2 \\ \varphi_3\\ \vdots \\ \varphi_n \end{smallmatrix}\right]}$, then 
$$
\begin{cases}
u_{i,j}\to -\varphi_{i-1}\ \mbox{in}\ X^{\frac{n-i+1}{n}}\hookrightarrow X^{\frac{n-i}{n}}\ \mbox{as}\ j\to\infty,\ \mbox{for}\ 2\leq i\leq n\\
A u_{1,j} \to \varphi_n\ \mbox{in}\ X\ \mbox{as}\ j\to\infty
\end{cases}
$$
and consequently, $u_i=-\varphi_{i-1}\in X^{\frac{n-i+1}{n}}$, for $2\leq i\leq n$. Finally, using the fact that $A$ is a closed operator, we have $u_1\in D(A)$ and $Au_1=\varphi_n$; that is, ${\bf u}\in D(\varLambda_n)$ and $\varLambda_n {\bf u}=\varphi$.

Secondly, $D(\varLambda_n)=D(A)\times D(A^{\frac{n-1}{n}})\times D(A^{\frac{n-2}{n}})\times\cdots\times D(A^{\frac{1}{n}})$ is dense in $Y=X^{\frac{n-1}{n}}\times X^{\frac{n-2}{n}}\times X^{\frac{n-3}{n}}\times\cdots\times X$ since the inclusions $X^{\alpha}\subset X^{\beta}$ are dense for $\alpha>\beta\geq 0$. 

Finally, since the operator $\varLambda_n$ is closed, $\lambda\in\rho(-\varLambda_n)$ if and only if the operator $\lambda I+\varLambda_n$ is bijective. From Lemma \ref{resoln} it follows easily that $[0,\infty)\subset \rho(-\varLambda_n)$. For ${\bf u}={\left[\begin{smallmatrix} u_1\\ u_2 \\ u_3\\ \vdots \\ u_n \end{smallmatrix}\right]}$, $\varphi={\left[\begin{smallmatrix} \varphi_1\\ \varphi_2 \\ \varphi_3\\ \vdots \\ \varphi_n \end{smallmatrix}\right]}$ in $Y$ and any $\lambda\geqslant0$ we have
\[
(\lambda I+\varLambda_n)^{-1}{\bf u}=\varphi,
\]
if and only if we have identities as \eqref{asad45tg} and \eqref{asad45tg2}.

Note that, for $1\leq i\leq n$, $\varphi_i$ is well defined since $\lambda^n\in\rho(-A)$. Moreover, $\varphi_i\in D(A^{\frac{n-i+1}{n}})$. In order to verify the equation \eqref{estres} for $\varLambda_n$,  it is sufficient to show that for $\|\mathbf{u}\|_Y\leq 1$ there exists a constant $K_{\varLambda_n}\geq 1$ such that
\begin{equation}
\|\varphi_1\|_{X^{\frac{n-1}{n}}}+\|\varphi_2\|_{X^{\frac{n-2}{n}}}+\|\varphi_3\|_{X^{\frac{n-3}{n}}}+\cdots+\|\varphi_n\|_{X} \leqslant \frac{K_{\varLambda_n}}{1+\lambda}
\end{equation}

Applying Lemma \ref{estpotb} with \eqref{asad45tg} and \eqref{asad45tg2} we obtain a constant $K\geqslant 1$ such that 
\[
\|\varphi_i\|_{X^{\frac{n-i}{n}}} \leqslant \frac{K_{\varLambda_n}}{1+\lambda},
\] 
for any $1\leqslant i\leqslant n$.
\qed 

\begin{remark}\label{sa3fbb}
Let $T:D(T)\subset E\to E$ be a linear operator  on some Banach space $E$, it is well known that if $-T$ generates a strongly continuous semigroup on $E$, this is also true for $-T^{\alpha}$ for $\alpha\in (0,1)$. However, what can one say about $-T^{\alpha}$ if $-T$ does not generate a strongly continuous semigroup? In general if $T$ is of positive type on $E$ (see Definition \ref{positive type}) then $-T^{\alpha}$ generates a strongly continuous analytic semigroup on $E$ for $0<\alpha\leq \frac{1}{2}$, see Amann \cite[Remark 4.6.12]{A}. 
\end{remark}
 
From now on  we study spectral properties of the fractional power operators $\varLambda_n^\alpha$ for any $0\leqslant\alpha\leqslant1$. 

A proof of the following result is given in Abramowitz and Stegun \cite{CA}.

\begin{lemma}
Let $U_n$ $n$th degree Chebyshev polynomial of the second kind defined in Theorem \ref{lem:useful:properties}. Then the following identities are holds.
\begin{equation}\label{A1as}
U_{n}(\cos \theta)\sin{\theta}=\sin((n+1)\theta),
\end{equation}
\begin{equation}\label{A3as}
U_n(-x)=(-1)^nU_n(x),
\end{equation}
and in particular,
\begin{equation}\label{A4as0}
U_n(1)=n+1,
\end{equation}
\begin{equation}\label{A4as}
U_n(-1)=(n+1)(-1)^n,
\end{equation}
for all $\theta\in\mathbb{R}$, $x\in\mathbb{C}$ and $n\geqslant0$.
\end{lemma}

\subsection{Proof of Theorem \ref{lem:useful:properties}: Fractional powers}

Part $i)$ immediately follows from the definition of $\varLambda_n$ and from fact that $\varLambda_n^{-1}$ takes bounded subsets of $Y$ into bounded subsets of $Y^1$, the latter space being compactly embedded in $Y$;

For part $ii)$ see \cite[Theorem $2.6.9$]{P};

Concerning part $iii)$ note that \eqref{asas324} is holds for $\alpha=0$ ($Y^0=Y$). Moreover, in this case, under  the  trigonometric identities, we have
\[
\frac{(-1)^{i-j}}{n}U_{n-1}\left(\cos\left(\frac{(i-j)\pi}{n}\right)\right)A^{\frac{i-j}{n}}=
\begin{cases}
I,& \mbox{if}\ i=j,\\
0,& \mbox{otherwise},
\end{cases}
\]
that is,
\[
 \left[\frac{(-1)^{i-j}}{n}U_{n-1}\left(\cos\left(\frac{(1+i-j)\pi}{n}\right)\right)A^{\frac{1+i-j}{n}}\right]_{ij}=\ \mbox{Identity operator on}\ Y,
\]
the identity \eqref{pota} is holds for $\alpha=0$.

Note that \eqref{asas324} is holds for $\alpha=1$, see \eqref{LamOp1}. Moreover, in this case, under  the  trigonometric identities, we have
\[
\frac{(-1)^{i-j}}{n}U_{n-1}\left(\cos\left(\frac{(1+i-j)\pi}{n}\right)\right)A^{\frac{1+i-j}{n}}=
\begin{cases}
-I,& \mbox{if}\ j=i+1,\\
A,& \mbox{if}\ i=n,\ j=1,\\
0,& \mbox{otherwise},
\end{cases}
\]
that is,
\[
 \left[\frac{(-1)^{i-j}}{n}U_{n-1}\left(\cos\left(\frac{(1+i-j)\pi}{n}\right)\right)A^{\frac{1+i-j}{n}}\right]_{ij}=\varLambda_n
\]
and consequently, by \eqref{LamOp2} the identity \eqref{pota} is holds for $\alpha=1$.

In the case $0<\alpha<1$, from \eqref{LamOp2} and \eqref{dasd4odd} we have
\begin{equation}\label{dasd4oddnew}
\varLambda_n(\lambda I+\varLambda_n)^{-1}=(\lambda^{n}I+A))^{-1} \cdot
 \left[\begin{smallmatrix}
A  &  -\lambda^{n-1} I  & -\lambda^{n-2}  I & \cdots & -\lambda^2 I & -\lambda I \\ \\ 
\lambda A & A  & -\lambda^{n-1} I & \cdots & -\lambda^3 I & -\lambda^2 I \\ \\ 
\lambda^2 A & \lambda A  & A & \cdots & -\lambda^4 I & -\lambda^3 I \\ \\ 
\vdots & \vdots & \vdots & \ddots & \vdots & \vdots\\ \\ 
\lambda^{n-2}A & \lambda^{n-3}A   &  \lambda^{n-4}A  & \cdots & A & -\lambda^{n-1} I \\ \\ 
\lambda^{n-1}A & \lambda^{n-2}A   &  \lambda^{n-3}A  & \cdots & \lambda A & A 
\end{smallmatrix}\right]
\end{equation}
for all $\lambda\in\rho(-\varLambda_n)$. 

In other words,
\begin{equation}
\varLambda_n(\lambda I+\varLambda_n)^{-1}=[e_{ij}],
\end{equation}
where
\[
e_{ij}=
\begin{cases}
\lambda^{i-j}A (\lambda^{n}I+A))^{-1},&\ \mbox{if}\ i\geqslant j;\\
-\lambda^{n+i-j}(\lambda^{n}I+A))^{-1},&\ \mbox{if}\ i<j;
\end{cases}
\]
for all $\lambda\in\rho(-\varLambda_n)$. 

Thus
\begin{equation}
\lambda^{\alpha-1}\varLambda_n(\lambda I+\varLambda_n)^{-1}=[E_{ij}],
\end{equation}
where
\[
E_{ij}=
\begin{cases}
\lambda^{\alpha+i-j-1}A ((-1)^{n}\lambda^{n}I+A))^{-1},&\ \mbox{if}\ i\geqslant j;\\
-\lambda^{\alpha+n+i-j-1}((-1)^{n}\lambda^{n}I+A))^{-1},&\ \mbox{if}\ i<j;\\
\end{cases}
\]
for all $\lambda\in\rho(-\varLambda_n)$ and $0<\alpha<1$. 

Now we apply in each entry $E_{ij}$ the fractional formula for $A$ given by \eqref{FracPower}, and after the change of variable $\mu=\lambda^{n}$ ($\lambda=\mu^{\frac{1}{n}}$ and $d\lambda=\frac{1}{n}\mu^{\frac{1-n}{n}}d\mu$) for $i\geqslant j$, we obtain the following: 
\begin{equation}\label{A1asneesa2}
\begin{split}
&\dfrac{\sin(\alpha \pi)}{\pi}\int_0^\infty   \lambda^{\alpha+i-j-1}A (\lambda^{n}I+A))^{-1}d\lambda\\
&=\dfrac{1}{n} \dfrac{\sin(\alpha \pi)}{\pi}\int_0^\infty \mu^{\frac{\alpha+i-j}{n}-1}A(\mu I+A)^{-1} d\mu\\
&=\dfrac{(-1)^{i-j}}{n}\dfrac{\sin\Big((\frac{\alpha+i-j}{n})n \pi\Big)}{\pi}\int_0^\infty \mu^{\frac{\alpha+i-j}{n}-1}A(\mu I+A)^{-1} d\mu,
\end{split}
\end{equation}
and by \eqref{A1as} we have 
\begin{equation}\label{A1asnee}
\sin\Big((\frac{\alpha+i-j}{n})n \pi\Big)=U_{n-1}\Big(\cos\Big((\frac{\alpha+i-j}{n}) \pi\Big)\Big)\sin\Big((\frac{\alpha+i-j}{n}) \pi\Big)
\end{equation}
and by \eqref{A1asneesa2}, \eqref{A1asnee} and $0<\alpha+i-j<n$, we have
\begin{equation}\label{A1as1214sa2}
(-1)^{i-j+1}\dfrac{\sin(\alpha \pi)}{\pi}\int_0^\infty   \lambda^{\alpha+i-j-1}A (\lambda^{n}I+A))^{-1}d\lambda=\dfrac{(-1)^{i-j}}{n}U_{n-1}\Big(\cos\Big((\frac{\alpha+i-j}{n}) \pi\Big)\Big)A^{\frac{\alpha+i-j}{n}}.
\end{equation}

Similarly, for $i< j$, we obtain the following:  
\begin{equation}\label{mlas335}
\begin{split}
&-\dfrac{\sin(\alpha \pi)}{\pi}\int_0^\infty  \lambda^{\alpha+n+i-j-1} (\lambda^{n}I+A))^{-1}d\lambda\\
&=-\dfrac{1}{n} \dfrac{\sin(\alpha \pi)}{\pi}\int_0^\infty \mu^{\frac{\alpha+i-j}{n}}(\mu I+A)^{-1} d\mu\\
&=\dfrac{(-1)^{i-j+n+1}}{n}\dfrac{\sin\Big((\frac{\alpha+i-j+n}{n})n \pi\Big)}{\pi}\int_0^\infty \mu^{\frac{\alpha+i-j+n}{n}-1}(\mu I+A)^{-1} d\mu,
\end{split}
\end{equation}
and by \eqref{A1as} we have 
\begin{equation}\label{434ff5y}
\begin{split}
\sin\Big((\frac{\alpha+i-j+n}{n})n \pi\Big)&=U_{n-1}\Big(\cos\Big((\frac{\alpha+i-j+n}{n}) \pi\Big)\Big)\sin\Big((\frac{\alpha+i-j+n}{n}) \pi\Big)\\
&=(-1)^{n-1}U_{n-1}\Big(\cos\Big((\frac{\alpha+i-j}{n}) \pi\Big)\Big)\sin\Big((\frac{\alpha+i-j+n}{n}) \pi\Big)
\end{split}
\end{equation}
and from \eqref{mlas335}, \eqref{434ff5y} and $0<\alpha+i-j+n<n$, we get
\begin{equation}\label{A1am92324}
(-1)^{n+i-j}\dfrac{\sin(\alpha \pi)}{\pi}\int_0^\infty  \lambda^{\alpha+n+i-j-1} (\lambda^{n}I+A))^{-1}d\lambda=\dfrac{(-1)^{i-j}}{n}U_{n-1}\Big(\cos\Big((\frac{\alpha+i-j}{n}) \pi\Big)\Big)A^{\frac{\alpha+i-j}{n}}.
\end{equation}
and thanks to \eqref{A1as1214sa2} and  \eqref{A1am92324} we get \eqref{pota}.

Finally, it is easy to see that the matrix representation of $\varLambda_n^\alpha$ is a Toeplitz matrix for any $0\leqslant\alpha\leqslant1$.
\qed

\begin{remark}
For $n=2$ we have the stationary wave-like operator, and  note that \eqref{asas324} is holds for $0<\alpha\leqslant1$, see \cite{BCCN}; namely
\[
\begin{split}
\varLambda_2^{\alpha}&= \left[\dfrac{(-1)^{i-j}}{2}U_{1}\left(\cos\left(\frac{(\alpha+i-j)\pi}{2}\right)\right)A^{\frac{\alpha+i-j}{2}}\right]_{ij}\\
&= \left[(-1)^{i-j}\cos\left(\frac{(\alpha+i-j)\pi}{2}\right)A^{\frac{\alpha+i-j}{2}}\right]_{ij}\\
&= \left[\begin{matrix}
\cos(\frac{\alpha\pi}{2})A^{\frac{\alpha}{2}} & -\sin(\frac{\alpha\pi}{2})A^{\frac{\alpha-1}{2}}\\
\sin(\frac{\alpha\pi}{2})A^{\frac{\alpha+1}{2}} & \cos(\frac{\alpha\pi}{2})A^{\frac{\alpha}{2}}
\end{matrix}\right];
\end{split}
\]

For $n=3$ we have the stationary Moore-Gibson-Thompson-like operator, and note that \eqref{asas324} is holds for $0<\alpha\leqslant1$, see \cite{BS}; namely
\[
\begin{split}
&\varLambda_3^{\alpha}= \dfrac{1}{3}\left[ (-1)^{i-j}U_{2}\left(cos\left(\frac{(\alpha+i-j)\pi}{3}\right)\right)A^{\frac{\alpha+i-j}{3}}\right]_{ij}\\
&=\dfrac{1}{3} \left[\begin{matrix}
(4\cos^2(\frac{\alpha\pi}{3})-1)A^{\frac{\alpha}{3}} & -(4\cos^2(\frac{(\alpha-1)\pi}{3})-1)A^{\frac{\alpha-1}{3}} & (4\cos^2(\frac{(\alpha-2)\pi}{3})-1)A^{\frac{\alpha-2}{3}}\\
-(4\cos^2(\frac{(\alpha+1)\pi}{3})-1)A^{\frac{\alpha+1}{3}}&(4\cos^2(\frac{\alpha\pi}{3})-1)A^{\frac{\alpha}{3}} & -(4\cos^2(\frac{(\alpha-1)\pi}{3})-1)A^{\frac{\alpha-1}{3}}\\
(4\cos^2(\frac{(\alpha+2)\pi}{3})-1)A^{\frac{\alpha+2}{3}} & -(4\cos^2(\frac{(\alpha+1)\pi}{3})-1)A^{\frac{\alpha+1}{3}} & (4\cos^2(\frac{\alpha\pi}{3})-1)A^{\frac{\alpha}{3}}
\end{matrix}\right]\\
&=\Big[\varLambda_{ij}\Big],
\end{split}
\]
where
\[
\varLambda_{ij}=
\begin{cases}
\frac{1}{3}(4\cos^2(\frac{\alpha\pi}{3})-1)A^{\frac{\alpha}{3}},&\mbox{if}\ i=j,\\
-\frac{1}{3}(2\sin^2(\frac{\alpha\pi}{3})+\sqrt{3}\sin(\frac{2\alpha\pi}{3}))A^{\frac{\alpha-1}{3}},&\mbox{if}\ j=i+1,\\
\frac{1}{3}(2\sin^2(\frac{\alpha\pi}{3})-\sqrt{3}\sin(\frac{2\alpha\pi}{3}))A^{\frac{\alpha-2}{3}},&\mbox{if}\ j=i+2,\\
-\frac{1}{3}(2\sin^2(\frac{\alpha\pi}{3})-\sqrt{3}\sin(\frac{2\alpha\pi}{3}))A^{\frac{\alpha-1}{3}},&\mbox{if}\ i=j-1,\\
\frac{1}{3}(2\sin^2(\frac{\alpha\pi}{3})+\sqrt{3}\sin(\frac{2\alpha\pi}{3}))A^{\frac{\alpha-2}{3}},&\mbox{if}\ i=j-2.
\end{cases}
\]
\end{remark}

\begin{definition}\label{positive type}
Let $E$ be a Banach space over a field $\mathbb{K}$ ($\mathbb{K}=\mathbb{R}$ or $\mathbb{K}=\mathbb{C}$), and let $T: D(T) \subset E \to E$ be a linear operator. We say that $T$ is  non-negative, in the sense of  Amann \cite{A} if $\rho(T)\supset [0,+\infty)$ and the set $\{\lambda (\lambda I+T)^{-1}\}$ is equicontinuous, i.e., 
\begin{equation}\label{estres}
\|(\lambda I+T)^{-1}\|_{\mathcal{L}(E)}\leqslant \frac{K}{\lambda}, \ \mbox{for all}\ \lambda\geqslant 0.
\end{equation}
\end{definition}

The next result is proved in \cite[Theorem 3.6]{MS}, see also \cite[Theorem 3.2]{MS} and \cite[Theorem 3.4]{MS}.

\begin{lemma}\label{LemmaMS}
Let $E$ be a Banach space over a field $\mathbb{K}$ ($\mathbb{K}=\mathbb{R}$ or $\mathbb{K}=\mathbb{C}$), and let $T: D(T) \subset E \to E$ be a linear closed operator densely defined. Assume that $T$ is non-negative, in the sense of Amann \cite{A}, and $\sigma(T)$ is not empty, then
\begin{equation}\label{ew34rf}
\sigma(T^\alpha)=\{z^\alpha; z\in\sigma(T)\}
\end{equation}
for any $0<\alpha<1$.  If $\sigma(T)$ is not empty, then the spectrum $\sigma(T)$ also is.
\end{lemma}

The following result may be established by a straightforward extension of \cite[Theorem 3.6]{MS} so we omit its proof.

\begin{lemma}\label{corpos222}
The unbounded linear operator $\varLambda_n$ defined in \eqref{LamOp1}-\eqref{LamOp2} is such that \eqref{ew34rf} holds for any $0<\alpha<1$.  
\end{lemma}

\subsection{Proof of Theorem \ref{0intervalowp0}: Eingenvalues}

Thanks to Lemma \ref{corpos222} we can conclude that for each $n\geqslant 2$ the linear operator $-\varLambda_n$ is non-negative, in the sense of Amann \cite{A}. We also know that $\sigma(-\varLambda_n)$ is not empty, and by \eqref{eq:neigenvaluesLimit} the spectrum of $-\varLambda_n$ is such that the point spectrum consisting of eigenvalues 
\[
\bigcup_{k=0}^{\lfloor \frac{n-1}{2}\rfloor}\left(\left\{\mu_{j}^{\frac{1}{n}}e^{ \frac{i\pi (2k+1)}{n}}:\ j\in\mathbb{N}\right\}\cup\left\{\mu_{j}^{\frac{1}{n}}e^{ \frac{i\pi (2n-2k-1)}{n}}:\ j\in\mathbb{N}\right\}\right)
\]
where $\{\mu_j\}_{j\in\mathbb{N}}$ denotes the ordered sequence of eigenvalues of $A$ including their multiplicity and $\lfloor x\rfloor := \max \{z\in \mathbb{Z}\ |\  z\leq x\}$ (this implies \eqref{eq:neigenvalues}  holds for $\alpha=1$). Consequently, by \eqref{eq:neigenvaluesLimit} and Lemma \ref{LemmaMS} we have for each $0<\alpha<1$ the spectrum of $\varLambda_n^{\alpha}$ is such that the point spectrum consisting of eigenvalues  
\[
\bigcup_{k=0}^{\lfloor \frac{n-1}{2}\rfloor}\left(\left\{\mu_{j}^{\frac{\alpha}{n}}e^{ -i \frac{\pi  (n-2k-1)\alpha}{n}}:\ j\in\mathbb{N}\right\}\cup\left\{\mu_{j}^{\frac{\alpha}{n}}e^{ i \frac{\pi (n-2k-1)\alpha}{n}}:\ j\in\mathbb{N}\right\}\right)
\]
where $\{\mu_j\}_{j\in\mathbb{N}}$ denotes the ordered sequence of eigenvalues of $A$ including their multiplicity and $\lfloor x\rfloor := \max \{z\in \mathbb{Z}\ |\  z\leq x\}$, which implies that for each $0<\alpha<1$ the spectrum of $-\varLambda_n^{\alpha}$ is such that the point spectrum consisting of eigenvalues  
\[
\bigcup_{k=0}^{\lfloor \frac{n-1}{2}\rfloor}\left(\left\{\mu_{j}^{\frac{\alpha}{n}}e^{ i \frac{\pi (n- (n-2k-1)\alpha)}{n}}:\ j\in\mathbb{N}\right\}\cup\left\{\mu_{j}^{\frac{\alpha}{n}}e^{ i \frac{\pi (n+(n-2k-1)\alpha)}{n}}:\ j\in\mathbb{N}\right\}\right)
\]
where $\{\mu_j\}_{j\in\mathbb{N}}$ denotes the ordered sequence of eigenvalues of $A$ including their multiplicity and $\lfloor x\rfloor := \max \{z\in \mathbb{Z}\ |\  z\leq x\}$
\qed

\begin{remark}
We include below figures which reflects, in particular, the loss of a ``good'' sectoriality property as $\alpha\nearrow \frac{3}{4}$ for $n=3$, the loss of a ``good'' sectoriality property as $\alpha\nearrow \frac{2}{3}$ for $n=4$, the loss of a ``good'' sectoriality property as $\alpha\nearrow \frac{5}{8}$ for $n=5$. 

Moreover, e.g. for $\frac{3}{4}<\alpha\leqslant1$ the figures \ref{fig01}, \ref{fig02} and \ref{fig03} reflects, in particular, the loss of well-posedness of the Cauchy problem \eqref{mainprob} for $n=3$; for $\frac{2}{3}<\alpha\leqslant1$ the figures \ref{fig04} and \ref{fig05} reflects, in particular, the loss of well-posedness of the Cauchy problem \eqref{mainprob} for $n=4$; for $\frac{5}{8}<\alpha\leqslant1$ the figures \ref{fig06} and \ref{fig07} reflects, in particular, the loss of well-posedness of the Cauchy problem \eqref{mainprob} for $n=5$;

\begin{figure}[H]
\begin{center}
\begin{tikzpicture}
\draw[-stealth', densely dotted] (0,0) -- (3.15,0) node[below] {\ \ \ \ \ \ \ $\scriptstyle {\rm Re} (\lambda)$};
\draw[-stealth', densely dotted ] (0,-3.15) -- (0,3.15) node[left] {\color{black}$\scriptstyle{\rm Im} (\lambda)$};
\draw[densely dotted ] (0.7,-3.15) -- (0.7,3.15);
\fill[gray!10!] (0.7,-3.15) rectangle (3.15,-0.02);
\fill[gray!10!] (0.7,0.02) rectangle (3.15,3.15);
\node at (-4.5,2.1) {{\tiny\color{blue} Semi-line containing the eigenvalues of $\varLambda_3^{\frac{3}{4}}$}};
\draw[color=blue,->, densely dotted] (-4,1.8) -- (-0.1,-1.3);
\draw[color=blue,->, densely dotted] (-4,1.8) -- (-0.1,1.3);
\draw[color=blue,->,densely dotted] (-4,1.8) -- (-1.5,0.1);
\draw[color=blue, line width=1pt] (0,3.15) -- (0,0);
\draw[color=blue, line width=1pt] (0,-3.15) -- (0,0);
\draw[color=blue, line width=1pt] (-3.15,0.02) -- (0,0.02);
\node at (0.5,-0.2) {\color{blue}{\tiny $\omega$}};
\node at (-3,-0.3) {\color{blue}{\tiny $\mu_n^{\frac{1}{4}}e^{i\pi}$}};
\node at (0.7,2.8) {\color{blue}{\tiny $\mu_n^{\frac{1}{4}}e^{i\frac{\pi}{2}}$}};
\node at (0.7,-2.8) {\color{blue}{\tiny $\mu_n^{\frac{1}{4}}e^{i\frac{3\pi}{2}}$}};
\end{tikzpicture}
\end{center}
\caption{Location of the eigenvalues for  $\alpha=\frac{3}{4}$ with $n=3$.}\label{fig02}
\end{figure}
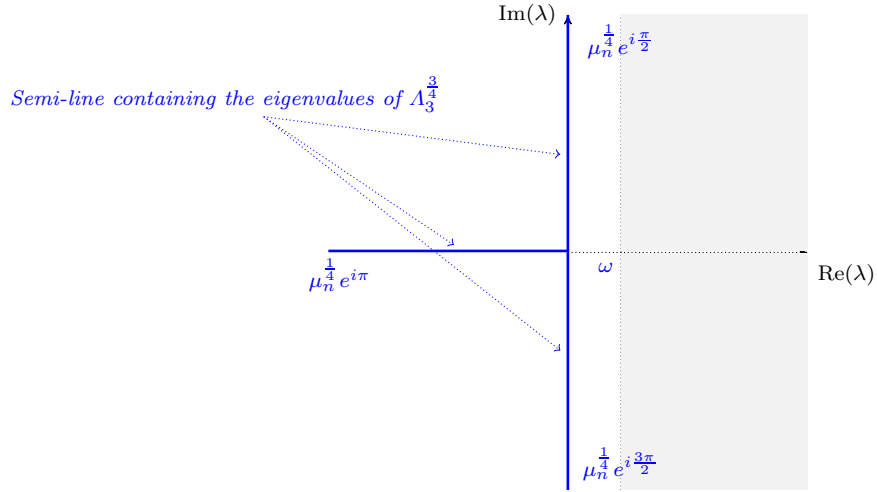

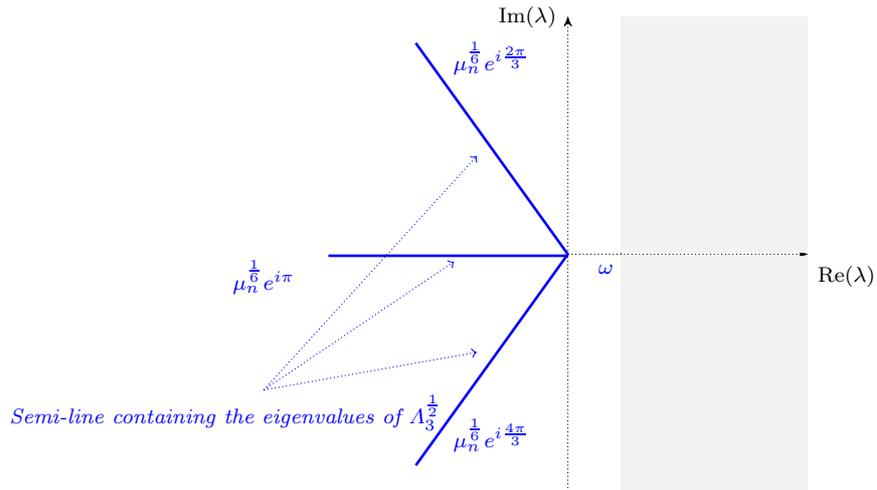
\begin{figure}[H]
\begin{center}
\begin{tikzpicture}
\draw[-stealth', densely dotted] (0,0) -- (3.15,0) node[below] {\ \ \ \ \ \ \ $\scriptstyle {\rm Re} (\lambda)$};
\draw[-stealth', densely dotted ] (0,-3.15) -- (0,3.15) node[left] {\color{black}$\scriptstyle{\rm Im} (\lambda)$};
\draw[densely dotted ] (0.7,-3.15) -- (0.7,3.15);
\fill[gray!10!] (0.7,-3.15) rectangle (3.15,-0.02);
\fill[gray!10!] (0.7,0.02) rectangle (3.15,3.15);
\node at (-4.5,-2.1) {{\tiny\color{blue} Semi-line containing the eigenvalues of $\varLambda_3^{\frac{1}{2}}$}};
\draw[color=blue,->, densely dotted] (-4,-1.8) -- (-1.2,-1.3);
\draw[color=blue,->, densely dotted] (-4,-1.8) -- (-1.2,1.3);
\draw[color=blue,->,densely dotted] (-4,-1.8) -- (-1.5,-0.1);
\draw[color=blue, line width=1pt] (-3.15,-0.02) -- (0,-0.02);
\draw[color=blue, line width=1pt] (-2,2.8) -- (0,0);
\draw[color=blue, line width=1pt] (-2,-2.8) -- (0,0);
\node at (0.5,-0.2) {\color{blue}{\tiny $\omega$}};
\node at (-1.0,2.6) {\color{blue}{\tiny $\mu_n^{\frac{1}{6}}e^{i\frac{2\pi}{3}}$}};
\node at (-4,-0.3) {\color{blue}{\tiny $\mu_n^{\frac{1}{6}}e^{i\pi}$}};
\node at (-1,-2.4) {\color{blue}{\tiny $\mu_n^{\frac{1}{6}}e^{i\frac{4\pi}{3}}$}};
\end{tikzpicture}
\end{center}
\caption{Location of the eigenvalues for  $\alpha=\frac{1}{2}$ with $n=3$.}\label{fig03}
\end{figure}

\begin{figure}[H]
\begin{center}
\begin{tikzpicture}
\draw[-stealth', densely dotted] (-3.15,0) -- (3.15,0) node[below] {\ \ \ \ \ \ \ $\scriptstyle {\rm Re} (\lambda)$};
\draw[-stealth', densely dotted ] (0,-3.15) -- (0,3.15) node[left] {\color{black}$\scriptstyle{\rm Im} (\lambda)$};
\draw[densely dotted ] (0.7,-3.15) -- (0.7,3.15);
\fill[gray!10!] (0.7,-3.15) rectangle (3.15,-0.02);
\fill[gray!10!] (0.7,0.02) rectangle (3.15,3.15);
\node at (-5,-2.1) {{\tiny\color{blue} Semi-line containing the eigenvalues of $\varLambda_4^{\frac23}$}};
\draw[color=blue,->, densely dotted] (-4,-1.8) -- (-1.4,-0.6);
\draw[color=blue,->, densely dotted] (-4,-1.8) -- (-0.2,-1.3);
\draw[color=blue,->, densely dotted] (-4,-1.8) -- (-1.4,0.6);
\draw[color=blue,->, densely dotted] (-4,-1.8) -- (-0.2,1.3);
\draw[color=blue, line width=1pt] (-2.5,1.3) -- (0,0);
\draw[color=blue, line width=1pt] (0,2.5) -- (0,0);
\draw[color=blue, line width=1pt] (-2.5,-1.3) -- (0,0);
\draw[color=blue, line width=1pt] (0,-2.5) -- (0,0);
\node at (0.5,-0.2) {\color{blue}{\tiny $\omega$}};
\node at (-1.6,1.4) {\color{blue}{\tiny $\mu_j^{\frac16}e^{i\frac{5\pi}{6}}$}};
\node at (0.6,2.6) {\color{blue}{\tiny $\mu_j^{\frac16}e^{i\frac{\pi}{2}}$}};
\node at (-1.6,-1.3) {\color{blue}{\tiny $\mu_j^{\frac16}e^{i\frac{7\pi}{6}}$}};
\node at (0.6,-2.4) {\color{blue}{\tiny $\mu_j^{\frac16}e^{i\frac{3\pi}{2}}$}};
\end{tikzpicture}
\end{center}
\caption{Location of the eigenvalues for  $\alpha=\frac{2}{3}$ with $n=4$.}\label{fig05}
\end{figure}
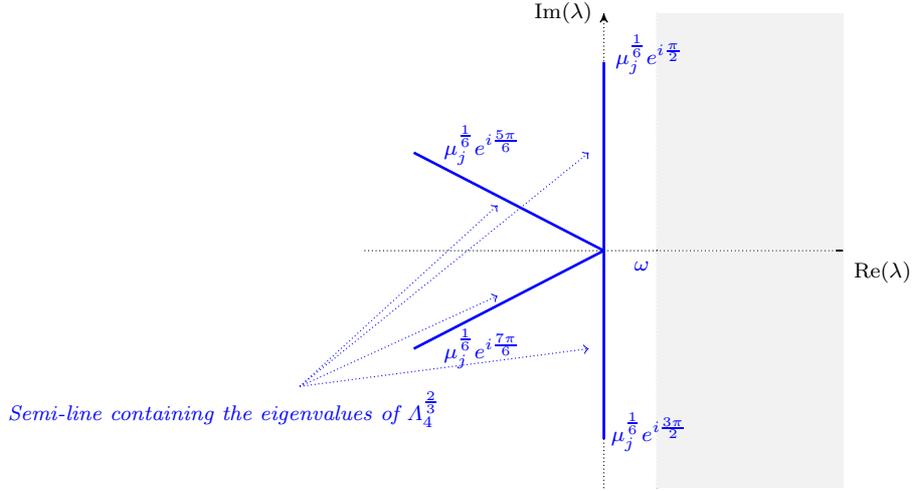

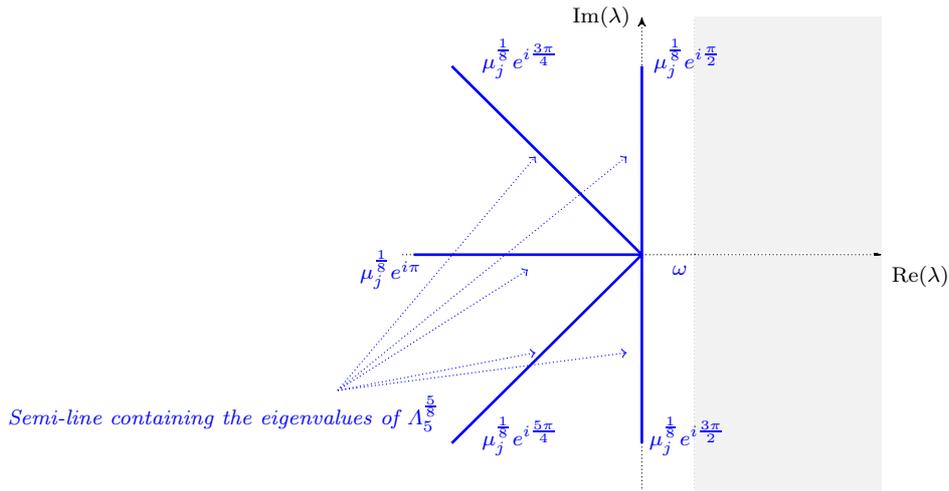
\begin{figure}[H]
\begin{center}
\begin{tikzpicture}
\draw[-stealth', densely dotted] (-3.15,0) -- (3.15,0) node[below] {\ \ \ \ \ \ \ $\scriptstyle {\rm Re} (\lambda)$};
\draw[-stealth', densely dotted ] (0,-3.15) -- (0,3.15) node[left] {\color{black}$\scriptstyle{\rm Im} (\lambda)$};
\draw[densely dotted ] (0.7,-3.15) -- (0.7,3.15);
\fill[gray!10!] (0.7,-3.15) rectangle (3.15,-0.02);
\fill[gray!10!] (0.7,0.02) rectangle (3.15,3.15);
\node at (-5.5,-2.1) {{\tiny\color{blue} Semi-line containing  the eigenvalues of $\varLambda_5^{\frac58}$}};
\draw[color=blue,->, densely dotted] (-4,-1.8) -- (-1.4,-1.3);
\draw[color=blue,->, densely dotted] (-4,-1.8) -- (-0.2,-1.3);
\draw[color=blue,->, densely dotted] (-4,-1.8) -- (-1.4,1.3);
\draw[color=blue,->, densely dotted] (-4,-1.8) -- (-0.2,1.3);
\draw[color=blue,->, densely dotted] (-4,-1.8) -- (-1.5,-0.2);
\draw[color=blue, line width=1pt] (-2.5,2.5) -- (0,0);
\draw[color=blue, line width=1pt] (0,2.5) -- (0,0);
\draw[color=blue, line width=1pt] (-3,0) -- (0,0);
\draw[color=blue, line width=1pt] (-2.5,-2.5) -- (0,0);
\draw[color=blue, line width=1pt] (0,-2.5) -- (0,0);
\node at (0.5,-0.2) {\color{blue}{\tiny $\omega$}};
\node at (-1.6,2.6) {\color{blue}{\tiny $\mu_j^{\frac18}e^{i\frac{3\pi}{4}}$}};
\node at (0.6,2.6) {\color{blue}{\tiny $\mu_j^{\frac18}e^{i\frac{\pi}{2}}$}};
\node at (-3.3,-0.2) {\color{blue}{\tiny $\mu_j^{\frac18}e^{i \pi}$}};
\node at (-1.6,-2.4) {\color{blue}{\tiny $\mu_j^{\frac18}e^{i\frac{5\pi}{4}}$}};
\node at (0.6,-2.4) {\color{blue}{\tiny $\mu_j^{\frac18}e^{i\frac{3\pi}{2}}$}};
\end{tikzpicture}
\end{center}
\caption{Location of the eigenvalues for  $\alpha=\frac{5}{8}$ with $n=5$.}\label{fig07}
\end{figure}
\end{remark}

The next result is proved in \cite[Theorem 3.4]{MS}.

\begin{lemma}\label{LemmaMS02}
Let $E$ be a Banach space over a field $\mathbb{K}$ ($\mathbb{K}=\mathbb{R}$ or $\mathbb{K}=\mathbb{C}$), and let $T: D(T) \subset E \to E$ be a linear closed operator densely defined. Assume that $T$ is non-negative, in the sense of Amann \cite{A}, then $0\in\rho(T)$ if and only if $0\in\rho(T^\alpha)$  for any $0<\alpha<1$. 
\end{lemma}

\section{Fractional differential equations}\label{Fractional differential equations}

In this section we recall the method of reduction of the coupled system \eqref{mainprobalpha} to an $n$th order differential equation. For this we use  Cayley-Hamilton-Ziebur Structure Theorem for $\frac{d{\bf x}}{dt}+M{\bf x}=0$, where $M$ is an $n\times n$ real or complex matrix.

\begin{theorem}
A component function $x_k(t)$ of the vector solution ${\bf x}(t)$ for $\frac{d{\bf x}}{dt}(t)+M{\bf x}(t)=0$ is a solution of the $n$th order linear homogeneous constant-coefficient differential equation whose characteristic equation is $\mbox{det}(M-rI)=0$.
\end{theorem}

It is well know that if we consider $M$ an $n\times n$ real or complex matrix, then the reduction of the coupled system
\[
\dfrac{d{\bf x}}{dt}+M{\bf x}=0,\quad t>0,
\]
to an $n$th order differential equation is given by 
\[
\dfrac{d^nx_1}{dt^n}+(-1)^n\mbox{tr}(M)\dfrac{d^{n-1}x_1}{dt^{n-1}}+\sum_{k=1}^{n-1}(-1)^k \mbox{tr}(\Lambda^k M) \dfrac{d^{n-k}x_1}{dt^{n-k}}+\mbox{det}(M)x_1=0,\quad t>0,
\]
where ${\bf x}(t)\left[\begin{smallmatrix} x_1\\ x_2 \\ x_3\\ \vdots \\ x_n \end{smallmatrix}\right](t)$, $x_{i+1}(\cdot)=\frac{d^ix_1}{dt^i}\Big|_{t=(\cdot)}$ for $i\in\{1,2,\ldots,n-1\}$, and
\begin{equation}\label{asa34tggg}
\mbox{tr}(\Lambda^k M) = \frac{1}{k!}  \mbox{det}
\begin{bmatrix}  
\mbox{tr}(M)  &   k-1 & 0 & \cdots & 0 \\
\mbox{tr}(M^2)  & \mbox{tr}(M) &  k-2 &\cdots & 0\\
 \vdots & \vdots & & \ddots & \vdots    \\
\mbox{tr}(M^{k-1}) & \mbox{tr}(M^{k-2}) & \mbox{tr}(M^{k-3}) & \cdots & 1    \\ 
\mbox{tr}(M^k)  &\mbox{tr}(M^{k-1}) & \mbox{tr}(M^{k-2}) & \cdots & \mbox{tr}(M)
\end{bmatrix},
\end{equation}
where  the matrix on the right side of \eqref{asa34tggg} is a $k\times k$ real or complex matrix.

From this and \eqref{eq:to-compute-fractional-powers}-\eqref{asas324} we obtain the following linear differential equations as ‘‘approximation fractional differential equations'' of \eqref{Eq1}
\begin{equation}\label{asa34tgasas}
\dfrac{d^nu^\alpha}{dt^n}+(-1)^n\mbox{tr}(\varLambda_n^{\alpha})\dfrac{d^{n-1}u^\alpha}{dt^{n-1}}+\sum_{k=1}^{n-1}(-1)^k \mbox{tr}(\Lambda^k \varLambda_n^{\alpha}) \dfrac{d^{n-k}u^\alpha}{dt^{n-k}}+\mbox{det}(\varLambda_n^{\alpha})u^\alpha=0,\quad t>0,
\end{equation}
where $0<\alpha\leqslant1$ and $\mbox{tr}(\varLambda_n^{\alpha})$ denotes the trace of the matrix representation of $\varLambda_n^{\alpha}$; namely
\[
\mbox{tr}(\varLambda_n^{\alpha}):=U_{n-1}\left(cos\left(\frac{\alpha\pi}{n}\right)\right)A^{\frac{\alpha}{n}},
\]
and $\mbox{det}(\varLambda_n^{\alpha})$ denotes the determinant of the matrix representation of $\varLambda_n^{\alpha} (=[a_{\alpha,ij}])$; namely for every $i$, one has the equality
\begin{equation}\label{asa34tgnh3}
\begin{split}
\mbox{det}(\varLambda_n^{\alpha})&:=\sum_{j=1}^n(-1)^{i+j}a_{\alpha,ij}M_{ij}\\
&=\frac{1}{n}\sum_{j=1}^n U_{n-1}\left(cos\left(\frac{(\alpha+i-j)\pi}{n}\right)\right)A^{\frac{\alpha+i-j}{n}}M_{\alpha,ij},
\end{split}
\end{equation}
where $M_{\alpha,ij}$ be the minor defined to be the determinant of the $(n-1)\times(n-1)$-matrix that results from $\varLambda_n^{\alpha}$ by removing the $i$th row and the $j$th column, and $\mbox{tr}(\Lambda^k \varLambda_n^{\alpha})$ is given by \eqref{asa34tggg}.

\begin{remark}
For example if $\alpha=1$ and $i=n$ in \eqref{asa34tgnh3}, then by \eqref{A1as} and \eqref{A4as} we have
\[
\mbox{tr}(\varLambda_n):=U_{n-1}\left(cos\left(\frac{\pi}{n}\right)\right)A^{\frac{1}{n}}=0,
\]
and
\[
\begin{split}
\mbox{det}(\varLambda_n)&=\frac{1}{n}\sum_{j=1}^nU_{n-1}\left(\cos\left(\frac{(1+n-j)\pi}{n}\right)\right)A^{\frac{1+n-j}{n}}M_{1,nj}\\
&=\frac{1}{n}U_{n-1}(-1) AM_{1,n1} +\frac{1}{n}\sum_{j=2}^nU_{n-1}\left(\cos\left(\frac{(1+n-j)\pi}{n}\right)\right)A^{\frac{1+n-j}{n}}M_{1,nj}\\
&=(-1)^{n-1}AM_{1,n1} +\frac{1}{n}\sum_{j=2}^nU_{n-1}\left(\cos\left(\frac{(1+n-j)\pi}{n}\right)\right)A^{\frac{1+n-j}{n}}M_{1,nj}\\
&=A +\frac{1}{n}\sum_{j=2}^nU_{n-1}\left(\cos\left(\frac{(1+n-j)\pi}{n}\right)\right)A^{\frac{1+n-j}{n}}M_{1,nj}\\
&=A,
\end{split}
\]
because $M_{1,n1}=(-1)^{n-1}$ and $\sin((1+n-j)\pi)=0$ for any $j\in\{2,3,\ldots,n\}$.
\end{remark}

By the Remark \ref{sa3fbb} we obtain linear parabolic equations \eqref{asa34tgasas} for any $0<\alpha<\frac{n}{2(n-1)}$.

\section{Applications}\label{Applications}

Let $\Omega\subset\mathbb{R}^N$ be a bounded domain with with sufficiently smooth boundary $\partial\Omega$,  $N\geqslant 1$, and let $X=L^2(\Omega)$ be endowed with the standard inner product. In this section we consider the unbounded linear operators $A_D:D(A_D)\subset X\to X$ defined by linear $2m$-th order uniformly elliptic partial differential operator
\begin{equation}\label{Operator1as}
A_D u=(-\Delta)^m u,\quad m\in\mathbb{N},
\end{equation}
with domain
\begin{equation}\label{Operator2as}
D(A_D)=H^{2m}(\Omega)\cap H_0^m(\Omega),
\end{equation}
and we also consider the linear  evolution equations of $n$-th order in time with $n\geqslant 3$ 
\begin{equation}\label{Eq1as}
\partial_t^nu+(-\Delta)^m u=0,
\end{equation}
subject to zero Dirichlet boundary condition and initial conditions
\begin{equation}\label{MGTe2}
\begin{cases}
u(x,t)=\Delta^j u(x,t)=0,& x\in\partial\Omega,\ t\geqslant0,\ j\in\{1,\ldots, m-1\},\\
u(x,0)=u_0(x),\ \partial_t^iu(x,0)=u_i(x),& x\in\partial\Omega,\ i\in\{1,\ldots, n-1\},
\end{cases}
\end{equation}
where $n\geqslant3$.

The unbounded linear operator $A_D$ defined in \eqref{Operator1as}-\eqref{Operator2as} is a closed, densely defined, self-adjoint and positive definite operator.  There exists $\zeta>0$ such that $\mbox{Re}\sigma(A_D)>\zeta$, that is,   $\mbox{Re}\lambda>\zeta$ for all $\lambda\in\sigma(A)$, and therefore, $A_D$ is a sectorial operator in the sense of Henry \cite[Definition 1.3.1]{H}, with the
eigenvalues $\{\nu_j\}_{j\in\mathbb{N}}$:
\[
0<\nu_1\leqslant\nu_{2}\leqslant\dots\leqslant\nu_j
\leqslant\dots,\quad \nu_j\to+\infty\quad \mbox{as}\ j\to+\infty.
\] 

This allows us to define the fractional power $A_D^{-\alpha}$ of order $\alpha\in(0,1)$ according to Amann \cite[Formula 4.6.9]{A} and Henry \cite[Theorem 1.4.2]{H}, as a closed linear operator on its domain $D(A_D^{-\alpha})$ with inverse $A_D^{\alpha}$. Denote by $X^{\alpha}=D(A_D^{\alpha})$ for $\alpha\in[0,1]$. The fractional power space $X^\alpha$ endowed with graphic norm 
\[
\|\cdot\|_{X^\alpha}:=\|A_D^{\alpha} \cdot\|_X
\] 
is a Banach space; namely, e.g., if $m\alpha$ is an integer, then 
\[
X^\alpha=D((-\Delta)^{m\alpha})=H^{2m\alpha}(\Omega)\cap H_0^{m\alpha}(\Omega)
\] 
with equivalent norms, see Cholewa and  D\l otko \cite[Page 29]{ChD} and Henry \cite[Pages 29 and 30]{H}. 

With this notation, we have $X^{-\alpha}=(X^\alpha)'$ for all $\alpha>0$, see Amann \cite{A} and Triebel \cite{T} for the characterization of the negative scale. 

The scale of fractional power spaces $\{X^\alpha\}_{\alpha\in\mathbb{R}}$ associated with $A_D$ safisty
\[
X^\alpha\subset H^{2m\alpha}(\Omega),\quad \alpha\in[0,1],
\]
where $H^{2m\alpha}(\Omega)$ are the potential Bessel spaces, see Cholewa and  D\l otko \cite[Page 48]{ChD}. 

From now on we study the case $m=1$. From Sobolev embedding theorem, we obtain
\[
X^\alpha\subset L^r(\Omega),\ \mbox{for any}\ r\leqslant\dfrac{2N}{N-4\alpha},\ 0\leqslant\alpha<\dfrac{N}{4},
\]
\[
X=L^2(\Omega),
\]
\[
L^s(\Omega)\subset X^\alpha,\ \mbox{for any}\ s\geqslant\dfrac{2N}{N-4\alpha},\ -\dfrac{N}{4}<\alpha\leqslant0,
\]
with continuous embeddings.  

It is not difficult to show that $A_D^{\alpha}$ is the generator of a strongly continuous analytic
semigroup on $X$, that we will denote by $\{e^{-tA_D^{\alpha}}:  t\geqslant 0\}$, see Kre\v{\i}n \cite{Kr} and  Tanabe \cite{Ta} for any $\alpha\in[0,1]$.  

Let $\alpha\in(0,1]$. We recall that the fractional powers of the negative Laplacian operator can be calculated through the spectral decomposition: since $X=L^2(\Omega)$ is a Hilbert space and $A_D=-\Delta$ with zero Dirichlet boundary condition in $\Omega$ is a self-adjoint operator and is the infinitesimal generator of a $C_0$-semigroup of contractions on $X$, it follows that there exists an orthonormal basis composed by eigenfunctions $\{\varphi_j\}_{j\in\mathbb{N}}$ of $A_D$. Let $\nu_j$ be the eigenvalues of $A_D=-\Delta$, then $(\nu^{\alpha}_j,\varphi_j)$ are the eigenvalues and eigenfunctions of $A_D^\alpha=(-\Delta)^\alpha$, also with zero Dirichlet boundary condition, respectively.

It is well know that the fractional Laplacian $A_D^\alpha:D(A_D^\alpha)\subset X\to X$ is well defined in the space
\[
D(A_D^\alpha)=X^\alpha=\Big\{u=\sum_{j=1}^{\infty} a_j\varphi_j\in L^2(\Omega):  \sum_{j=1}^{\infty} a^2_j\nu_j^{2\alpha} < \infty \Big\},
\]
where
\[
A_D^\alpha u=\sum_{j=1}^{\infty} \nu_j^{\alpha}a_j\varphi_j,\quad \ u \in D(A_D^\alpha) = X^\alpha.
\]

Finally, we apply all our results from previous sections to boundary value problem \eqref{Eq1as}-\eqref{MGTe2} to obtain a track in $\alpha$ in which we can present a result of solubility and passage to the limit at $\alpha\nearrow\frac{n}{2(n-1)}$ for fractional problems associated with  \eqref{Eq1as}-\eqref{MGTe2}.

Moreover, it is possible to consider other elliptic operators under other boundary conditions with behavior spectral as the negative laplacian operator under zero boundary conditions.

\end{document}